\documentclass[12pt, twoside]{article}
\usepackage{amsmath,amsthm,amssymb,color}
\usepackage{times}
\usepackage{enumerate}
\usepackage{empheq}
\usepackage{url}
\usepackage{lineno}
\usepackage{microtype}
\usepackage{cases}
\usepackage{times}
\usepackage{graphicx}
\def\titlerunning#1{\gdef\titrun{#1}}
\makeatletter
\def\author#1{\gdef\autrun{\def\and{\unskip, }#1}\gdef\@author{#1}}
\def\address#1{{\def\and{\\\hspace*{18pt}}\renewcommand{\thefootnote}{}%
		\footnote {#1}}%
	\markboth{\autrun}{\titrun}}
\makeatother
\def\email#1{e-mail: #1}
\def\subjclass#1{{\renewcommand{\thefootnote}{}%
		\footnote{\emph{Mathematics Subject Classification (2020):} #1}}}
\def\keywords#1{\par\medskip
	\noindent\textbf{Keywords.} #1}

\newtheorem{thm}{Theorem}
\newtheorem{prop}{Proposition}
\newtheorem{lem}{Lemma}
\newtheorem{de}{Definition}
\newtheorem{cor}{Corollary}
\newtheorem{re}{Remark}
\newtheorem{ex}{Example}
\numberwithin{equation}{section}

\newcommand{\R}{\mathbb{R}}
\newcommand{\va}{\varphi}
\newcommand{\eps}{\varepsilon}
\newcommand{\x}{\mathfrak{c}_l}
\newcommand{\y}{\mathfrak{c}_r}

\numberwithin{equation}{section}
\numberwithin{footnote}{section}

\begin{document}
	
	
	\baselineskip=16pt
	
	
	\titlerunning{Ergodic problems for contact Hamilton-Jacobi equations}
	
	\title{Ergodic problems for contact Hamilton-Jacobi equations}
	
	\author{Kaizhi Wang  \and Jun Yan}
	
	\date{\today}
	
	\maketitle
	
	\address{Kaizhi Wang: School of Mathematical Sciences, CMA-Shanghai, Shanghai Jiao Tong University, Shanghai 200240, China; \email{kzwang@sjtu.edu.cn}
		\and Jun Yan: School of Mathematical Sciences, Fudan University, Shanghai 200433, China;
		\email{yanjun@fudan.edu.cn}}


	\begin{abstract}
		This paper deals with the generalized ergodic problem
		\[
		H(x,u(x),Du(x))=c,	\quad x\in M,
		\]	
		where the unknown is a pair $(c,u)$ of a constant $c \in \mathbb{R}$ and a function $u$ on $M$ for which $u$ is a  viscosity solution. We assume $H=H(x,u,p)$ satisfies Tonelli conditions in the argument $p\in T^*_xM$ and the Lipschitz condition in the argument $u\in\R$.

		For a given $c\in \R$, we first discuss necessary and sufficient conditions for the existence of viscosity solutions. Let $\mathfrak{C}$ denote the set of all real numbers $c$'s for which the above equation admits viscosity solutions. Then we show $\mathfrak{C}$ is an interval, whose endpoints $\x$, $\y$ with $\x\leqslant\y$ can be  characterized by a min-max formula and a max-min formula, respectively.

		The most significant finding  is that we figure out the structure of $\mathfrak{C}$ without monotonicity assumptions on $u$.

		\medskip
	\end{abstract}

	\newpage
	\tableofcontents

	\section{Introduction}
	\renewcommand{\thethm}{\Alph{thm}}
	\setcounter{equation}{0}
	
	\subsection{Assumptions and main results}

	Let $M$ be a closed (compact, without boundary), connected and smooth manifold.
	Denote by $TM$ its tangent bundle and $T^*M$ the cotangent one. $\R$ stands for 1-dimensional real Euclidean space and $\R_+=\{x\in\R: x>0\}$.
	Let $H=H(x,u,p)$ be a $C^3$ function on $T^*M\times \mathbb{R}$ with $(x,p)\in T^*M$ and $u\in\R$, satisfying
	\begin{itemize}
		\item [\textbf{(H1)}] Strict convexity:  the second partial derivative $\frac{\partial^2 H}{\partial p^2} (x,u,p)$ is positive definite as a quadratic form for all $(x,u,p)\in T^*M\times\mathbb{R}$;
		\item [\textbf{(H2)}] Superlinearity:  $H(x,u,p)$ is  superlinear in $p$ for all $(x,u)\in M\times\mathbb{R}$;
		\item [\textbf{(H3)}] Lipschitz continuity: there exists $\lambda>0$ such that $|\frac{\partial H}{\partial u}(x,u,p)|\leqslant \lambda$ for all $(x,u,p)\in T^*M\times\mathbb{R}$.
	\end{itemize}
	\medskip
	
	Consider the contact Hamilton-Jacobi equation
	\begin{align}\tag{$E_c$}\label{shjc}
		H(x,u(x),Du(x))=c,	\quad x\in M.
	\end{align}
	The symbol $D$ in equation \eqref{shjc}  denotes the spatial gradient.

	Let
\begin{align*}
	\mathfrak{C}:=\Big\{c\in\mathbb{R}\ :\ \text{equation}\  \eqref{shjc}\ \text{has viscosity solutions}\Big\}.
\end{align*}
	
	We get two main results in this paper:
	\begin{itemize}
		\item Theorem \ref{thA} provides a set of necessary and sufficient conditions for the existence of viscosity solutions of \eqref{shjc} for any given $c\in\R$.
		\item Theorem \ref{thC} shows that $\mathfrak{C}$ is an interval with the left endpoint $\x$ and the right endpoint $\y$,
	where $\x\in[-\infty,+\infty)$ and $\y\in(-\infty,+\infty]$. Moreover, $\mathfrak{C}$ may be an open interval, a closed  interval or a half-open interval. Furthermore, we give a min-max formula for $\x$ and a max-min formula for $\y$.
	\end{itemize}

	Before stating the main results, we recall the key tools used in this paper---solution semigroups first. The contact Lagrangian $L(x,u,\dot{x})$ associated with $H(x,u,p)$ is defined by
	\[
	L(x,u, \dot{x}):=\sup_{p\in T^*_xM}\{\langle \dot{x},p\rangle_x-H(x,u,p)\}, \quad (x,\dot{x})\in TM,\ u\in\mathbb{R}.
	\]
Under assumptions (H1)-(H3) the authors of \cite{WWY2} introduced two semigroups of operators associated with the contact Lagrangian $L$, denoted by $\{T^-_t\}_{t\geqslant 0}$ and $\{T^+_t\}_{t\geqslant 0}$.
	For each $\varphi\in C(M,\R)$, denote by $(x,t)\mapsto T^-_t\varphi(x)$ the unique continuous function on $ (x,t)\in M\times[0,+\infty)$ such that
	\[
	T^-_t\varphi(x)=\inf_{\gamma}\left\{\varphi(\gamma(0))+\int_0^tL\left(\gamma(\tau),T^-_\tau\varphi(\gamma(\tau)),\dot{\gamma}(\tau)\right)d\tau\right\},
	\]
	where the infimum is taken among  curves $\gamma\in C^{ac}([0,t],M)$ with $\gamma(t)=x$. We call $\{T^-_t\}_{t\geqslant 0}$ the backward solution semigroup for
	equation
	\begin{align}\label{Cau}
		w_t(x,t)+H(x,w(x,t),Dw(x,t))=0.
	\end{align}
	The function $(x,t)\mapsto T^-_t\varphi(x)$ is the unique viscosity solution of equation \eqref{Cau} with the initial condition $w(x,0)=\va(x)$. Similarly, one can define another semigroup of operators $\{T^+_t\}_{t\geqslant 0}$, called the forward solution semigroup
	by
	\begin{equation*}\label{2-4}
		T^+_t\varphi(x)=\sup_{\gamma}\left\{\varphi(\gamma(t))-\int_0^tL(\gamma(\tau),T^+_{t-\tau}\varphi(\gamma(\tau)),\dot{\gamma}(\tau))d\tau\right\},
	\end{equation*}
	where the supremum is taken among curves $\gamma\in C^{ac}([0,t],M)$ with $\gamma(0)=x$.  We use $\{T^{-,c}_t\}_{t\geqslant 0}$ (resp. $\{T^{+,c}_t\}_{t\geqslant 0}$) to denote the backward (resp. forward) solution semigroup associated with $L+c$, where $c\in\R$.
			
			\medskip
			
			\noindent $\bullet$ {\bf Existence of viscosity solutions of \eqref{shjc}}.
			
			\medskip
			Let us recall the additive eigenvalue problem (or ergodic problem):
			let $G$ be a Hamiltonian defined on $T^*M$.
			Finding solutions $(c,u)$ of equation $G(x,Du(x))=c$ is a well-known problem, called the cell (or, corrector) problem. Under a set of standard assumptions, the real number $c$ is unique, for which the equation has viscosity solutions. Let $\mathfrak{G}_t$ be the solution operator of the corresponding evolutionary equation $w_t(x,t)+G(x,Dw(x,t))=0$. Then $u$ is a viscosity solution of $G(x,Du(x))=c$ if and only if $\mathfrak{G}_tu=u-ct$ for all $t\geqslant 0$. Since the relation $\mathfrak{G}_tu=u-ct$
			looks like a nonlinear eigenvalue problem, finding solutions $(c,u)$ of $G(x,Du(x))=c$ is also called an additive eigenvalue problem. The additive eigenvalue $c$ determines the effective Hamiltonian in the homogenization of Hamitlon-Jacobi equations \cite{E,LPV}. An interesting dynamical feature of $c$ was discovered by weak KAM theory for Tonelli Lagrangians \cite{Fat97a,Fat97b,Fat98,Fat-b}, where $c$ is called Ma\~n\'e critical value of $G$, and a link between viscosity solutions of $G(x,Du(x))=c$ and Aubry sets, Mather sets of Hamiltonian systems generated by $G$ was established. Under Tonelli conditions, Contreras et al. \cite{Ci} provided a representation formula for $c$\ :
			\[
			c=\inf_{f\in C^\infty(M)}\sup_{x\in M}G(x,Df(x)).
			\]
			The above infimum is not a minimum.
			This formula still holds true when $C^\infty(M)$ is replaced by $C^{1,1}(M)$, $C^{1}(M)$, or $\mathrm{Lip}(M)$ and $\inf$ is replaced by $\min$, see \cite{Be,FS}.

			Now come back to our problem \eqref{shjc}, which we call it {\em generalized additive eigenvalue problem (or generalized ergodic problem)}. When $H$ satisfies (H1), (H2) and  $0<\delta\leqslant \frac{\partial H}{\partial u}\leqslant \lambda$, it is well-known that \eqref{shjc} has viscosity solutions for each real number $c$. When $H$ satisfies (H1), (H2) and  $0\leqslant \frac{\partial H}{\partial u}\leqslant \lambda$, \eqref{shjc} has viscosity solutions if and only if there is $a\in\R$ such that Ma\~n\'e critical value of $H(x,a,p)$ is $c$. See \cite{WWY3} for an example where the range of the function $a\mapsto \text{Ma\~n\'e critical value of}\ H(x,a,p)$ is a proper subset of $\R$, which means that there exists $c\in\R$ such that \eqref{shjc} has no viscosity solutions.
			Seen in this light, studying the generalized additive eigenvalue problem \eqref{shjc} under (H1)-(H3) is not a straightforward task at all.
			\begin{thm}\label{thA}
				Let $c\in\R$. The following  statements are equivalent.
				\begin{itemize}
					\item [(1)] Equation \eqref{shjc} has viscosity solutions;
					\item [(2)] There exist $\varphi$, $\psi\in C(M,\R)$ and $t_1$, $t_2\in \R_+$ such that $T^{+,c}_{t_1}\varphi\leqslant \varphi$, $T^{+,c}_{t_2}\psi\geqslant \psi$;
					\item [(3)] There exist $\varphi$, $\psi\in C(M,\R)$ and $t_1$, $t_2\in \R_+$ such that $T^{-,c}_{t_1}\varphi\geqslant \varphi$, $T^{-,c}_{t_2}\psi\leqslant \psi$;
					\item [(4)] There exist $\varphi$, $\psi\in C(M,\R)$ such that $T^{-,c}_{t}\varphi$ is bounded from below and $T^{-,c}_{t}\psi$ is bounded from above on $M\times[0,+\infty)$.
				\end{itemize}
			\end{thm}
			
			\medskip
			

			\noindent $\bullet$ {\bf Structure of the set  $\mathfrak{C}$}.
			
			\medskip

			We call $\mathfrak{C}$ {\em the admissible set for the generalized ergodic problem} \eqref{shjc}.
			Under the same assumptions imposed in this paper, the existence of solutions $(c,u)$ of \eqref{shjc} was proven in \cite{WWY2}, i.e, $\mathfrak{C}\neq \emptyset$. But, the structure of the set $\mathfrak{C}$ was not discussed there.

			$\mathrm{SCL^-}(M)$ (resp. $\mathrm{SCL^+}(M)$) stands for the set of all functions which are semiconcave (resp. semiconvex) on $M$ with a linear modulus. $\mathrm{Lip}(M)$ stands for the space of Lipschitz continuous functions on $M$. Since $\mathrm{SCL^{\pm}}(M)\subset\mathrm{Lip}(M)$, then by Rademacher's theorem $Du(x)$ exists almost everywhere for each $u\in \mathrm{SCL^{\pm}}(M)$. See for example, \cite{CS} for more about semiconcave and semiconvex  functions. Let $\mathrm{Dom}(Du)$ denote the domain of definition of $Du$. We attempt to characterize the set $\mathfrak{C}$ using the following two constants (may be $\pm \infty$) determined by $H$.  Define
			\begin{align*}\label{ccc}
				\begin{split}
					\x: & =\inf_{u\in\mathrm{SCL^+}(M)}\sup_{x\in\mathrm{Dom}(Du)}H(x,u(x),Du(x)),\\[3mm]
					\y: & =\sup_{u\in \mathrm{SCL^+}(M)}\inf_{x\in \mathrm{Dom}(Du)}H(x,u(x),Du(x)).
				\end{split}
			\end{align*}
			
			\medskip
			The following result gives a complete answer to the structure problem for $\mathfrak{C}$.
			\begin{thm}\label{thC}
				\[
				(\x,\y)\subset\mathfrak{C}\subset[\x,\y].
				\]
			\end{thm}
			\medskip
			
			\begin{re} Let us take a closer look at $\mathfrak{C}$.
				\begin{itemize}
					\item [$\star$] The interval $\mathfrak{C}$ will be one of the following: $(\mathfrak{c}_l,\mathfrak{c}_r)$, $[\mathfrak{c}_l,\mathfrak{c}_r]$,
					$(\mathfrak{c}_l,\mathfrak{c}_r]$,
					$[\mathfrak{c}_l,\mathfrak{c}_r)$.
					More precisely, each case can happen. Let $H(x,u,p):=\|p\|_x^2+f(u)$, where $f(u)$ is a smooth function on $\R$ with $|f'(u)|\leqslant \lambda$. If $\mathrm{Ran}(f)=(a,b)$, then $\mathfrak{C}=(a,b)$. If $\mathrm{Ran}(f)=[a,b]$, then $\mathfrak{C}=[a,b]$. If $\mathrm{Ran}(f)=(a,b]$, then $\mathfrak{C}=(a,b]$.
					If $\mathrm{Ran}(f)=[a,b)$, then $\mathfrak{C}=[a,b)$. Here, $a$, $b\in\R$ with $a\leqslant b$, and $\mathrm{Ran}(f)$ denotes the range of $f$.		
					\item [$\star$]
					We will prove in Section 4 that there is no viscosity solutions of  \eqref{shjc}  either when $c<\x$, or when $c>\y$.
					So, in view of the non-emptiness of $\mathfrak{C}$, one can deduce that $\x\leqslant\y$.
					\item [$\star$] If $H$ does not depend on $u$, then by classical results we deduce that
					\[
					\mathfrak{C}=\{\text{Ma\~n\'e critical value of}\ H\}
					\]
					is a singleton and thus $\x=\y$.
					\item [$\star$] We will show what $\x$, $\y$ are in several examples in Section 4.
				\end{itemize}
			\end{re}

			\medskip

			\subsection{Historical remarks}
			Hamilton-Jacobi equations have been first introduced in classical mechanics, but find applications in many other fields of mathematics. The theory of viscosity solutions of Hamilton-Jacobi equations was introduced in the early 80's by Crandall and Lions \cite{CL1}, Crandall, Evans and Lions \cite{CL2}. It provides a suitable PDE framework for studying Hamilton-Jacobi equations which does not have classical solutions. For a good introductory book on viscosity solutions, we refer readers to \cite{Bar}. The theory of viscosity solutions has been extensively studied and refined by many authors, and, among the numerous contributions in the literature, we would like to point out that the weak KAM theory opened a way to study viscosity solutions of Hamilton-Jacobi equations with Tonelli Hamiltonians using the dynamical information of action minimizing orbits of Hamiltonian systems. We refer readers to \cite{Ar,Be2,Con,DS,FS,FR,FG,Gom,ishii,Ka,MV,S,Tr,WY} and the references therein for more details on this topic. Along this line, it is natural to consider whether one can use weak KAM type results for contact Hamiltonian systems to study viscosity solutions of contact Hamilton-Jacobi equations \eqref{shjc}.
			For weak KAM aspects for contact Hamiltonian systems, we refer readers to \cite{DFIZ,MS,Ms,WWY3,WWY4}.
			Under assumptions imposed in this paper, it was shown in \cite{WWY2} the existence of solutions $(c,u)$ of equation \eqref{shjc}. See \cite{j} for a similar result using traditional PDE methods.
			We aim to refine and deepen the results in \cite{WWY2} in the present paper. We still use dynamical tools from the weak KAM theory for contact Hamiltonian systems satisfying (H1)-(H3). These assumptions will be weakened in a forthcoming paper.

			Our method is dynamical in nature and inspired by the deep connection between contact Hamilton-Jacobi equations \eqref{shjc} and contact Hamiltonian system
			\begin{align}\label{c}
				\left\{
				\begin{array}{l}
					\dot{x}=\frac{\partial H}{\partial p}(x,u,p),\\[2mm]
					\dot{p}=-\frac{\partial H}{\partial x}(x,u,p)-\frac{\partial H}{\partial u}(x,u,p)p,\\[2mm]
					\dot{u}=\frac{\partial H}{\partial p}(x,u,p)\cdot p-H(x,u,p).
				\end{array}
				\right.
			\end{align}
			The authors of \cite{WWY1,WWY2,WWY3,WWY4} discussed the weak KAM \cite{Fat-b} and Aubry-Mather \cite{M,S1} aspects of contact Hamiltonian systems from variational principles, dynamical properties of action minimizing orbits, to weak KAM solutions, viscosity solutions of stationary and evolutionary contact Hamilton-Jacobi equations. Contact Hamiltonian systems have deep connection with contact topology and non-equilibrium
			thermodynamics, see for example, \cite{EP}.

			\subsection{Notations}
			We write as follows a list of symbols used throughout this paper.
			\begin{itemize}
				
				\item We choose, once and for all, a $C^{\infty}$ Riemannian metric on $M$. It is classical that there is a canonical way to associate to it a Riemannian metric on $T M .$ We use the same symbol $d$ to denote the distance function defined by the Riemannian metric on $M$ and the distance function defined by the Riemannian metric on $TM .$ We use the same symbol $\|\cdot\|_x$ to denote the
				norms induced by the Riemannian metrics on $T_xM$ and $T^*_xM$ for $x\in M$, and by $\langle \cdot,\cdot\rangle_x$ the canonical pairing between the tangent space $T_xM$ and the cotangent space $T^*_xM$.
				\item  $C^{k}\left(M,\R\right)(k \in \mathbb{N})$ stands for the function space of $k$ -times continuously differentiable functions on $M$, and $C^{\infty}\left(M,\R\right):=\bigcap_{k=0}^{\infty} C^{k}\left(M,\R\right)$. And $\|\cdot\|_\infty$ denotes the supremum norm on these spaces.
				\item
				$C^{ac}([a,b],M)$ stands for the space of absolutely continuous curves  $[a,b]\to M$.
				\item
				$\mathrm{Lip}(M)$ stands for the space of Lipschitz continuous functions on $M$.
				\item  Denote by $\mathrm{SCL^-}(M)$ (resp. $\mathrm{SCL^+}(M)$) the set of all functions which are semiconcave (resp. semiconvex) on $M$ with a linear modulus.
				\item  $Du(x)=(\frac{\partial u}{\partial x_1},\dots,\frac{\partial u}{\partial x_n})$ and $ Dw(x,t)=(\frac{\partial w}{\partial x_1},\dots,\frac{\partial w}{\partial x_n})$.
				\item $D^+u(x)$ denotes the Fr\'echet superdifferential of $u$
				at $x$.
				\item Let  $u\in \mathrm{Lip}(M)$. Denote by $\mathrm{Dom}(Du)$ the set of all points $x\in M$ where $Du(x)$ exists.
				\item $\mathcal{S}_-$ (resp. $\mathcal{S}_+$) denotes the set of all  backward (resp. forward) weak KAM solutions of equation $H(x,u(x),Du(x))=0$.
				\item Let $\Phi_t$ denote the local flow of contact Hamiltonian system (\ref{c}).
				\item $\mathrm{cl}(A)$ denotes the closure of a set $A\subset T^*M\times\R$.
				\item $\mathrm{co}(A)$ denotes the convex hull of a set $A\subset T^*M\times\R$.
				\item $h_{x_0,u_0}(x,t)$ (resp. $h^{x_0,u_0}(x,t)$) denotes the forward (resp. backward) implicit action function associated with $L$.
				\item $h^c_{x_0,u_0}(x,t)$ (resp. $h_c^{x_0,u_0}(x,t)$) denotes the forward (resp. backward) implicit action function associated with $L+c$, where $c\in\R$.
				\item $\{T^{-}_t\}_{t\geqslant 0}$ (resp. $\{T^{+}_t\}_{t\geqslant 0}$) denotes the backward (resp. forward) solution semigroup associated with $L$.
				\item $\{T^{-,c}_t\}_{t\geqslant 0}$ (resp. $\{T^{+,c}_t\}_{t\geqslant 0}$) denotes the backward (resp. forward) solution semigroup associated with $L+c$, where $c\in\R$.
				
				\item We use $\va$, $\psi$, $\va'$, $\psi'$to denote generic functions in $C(M,\R)$ not necessarily the same in any two proofs, and use $\va_\infty$, $\psi_\infty$, $\va'_\infty$, $\psi'_\infty$ to denote limit functions of $T^{\pm,c}_tf$ as $t\to+\infty$, with $f=\va$, $\psi$, $\va'$, $\psi'$.
			\end{itemize}
			
			The rest of this paper is organized as follows.
			In Section 2, we first recall some known results on the weak KAM theory for contact Hamiltonian systems, and then prove several new results on the solution semigroups which will be used later. In Section 3 we show Proposition \ref{thG} first. Then we prove Proposition \ref{thD}. Using the results obtained in Proposition \ref{thD} we prove Proposition \ref{thH}. Theorem \ref{thA} is a direct consequence of Proposition \ref{thG} and Proposition \ref{thH}. Section 4 is devoted to the proof of Theorem \ref{thC}.
			The proofs of some preliminary results are given in the Appendix.
			

			
			\medskip


			\section{Preliminaries}

			\subsection{Weak KAM theory for contact Hamiltonian systems}
			We recall some definitions and basic results in the weak KAM theory for contact Hamiltonian system (\ref{c}), where an implicit variational principle plays an essential role. Most of the results in this section can be found in  \cite{WWY1,WWY2,WWY3,WWY4}.

				Since $H$ satisfies (H1)-(H3), it is direct to check that $L$ satisfies: (L1) Strict convexity:  the second partial derivative $\frac{\partial^2 L}{\partial \dot{x}^2} (x,u,\dot{x})$ is positive definite as a quadratic form for all $(x,u,\dot{x})\in TM\times\mathbb{R}$; (L2) Superlinearity:  $L(x,u,\dot{x})$ is  superlinear in $\dot{x}$ for all $(x,u)\in M\times\mathbb{R}$; (L3) Lipschitz continuity: there exists $\lambda>0$ such that $|\frac{\partial L}{\partial u}(x,u,\dot{x})|\leq \lambda$ for all $(x,u,\dot{x})\in TM\times\mathbb{R}$.

				\medskip

			\noindent $\bullet$ {\bf Variational principles}.
			First recall implicit variational principles for contact Hamiltonian system (\ref{c}), which connects contact Hamilton-Jacobi equations and contact Hamiltonian systems.
			\begin{prop}\label{IVP}
				For any given $x_0\in M$, $u_0\in\mathbb{R}$, there exist two continuous functions $h_{x_0,u_0}(x,t)$ and $h^{x_0,u_0}(x,t)$ defined on $M\times(0,+\infty)$ satisfying	
				\begin{align}
					h_{x_0,u_0}(x,t)&=u_0+\inf_{\substack{\gamma(0)=x_0 \\  \gamma(t)=x} }\int_0^tL\big(\gamma(\tau),h_{x_0,u_0}(\gamma(\tau),\tau),\dot{\gamma}(\tau)\big)d\tau,\label{2-1}\\
					h^{x_0,u_0}(x,t)&=u_0-\inf_{\substack{\gamma(t)=x_0 \\  \gamma(0)=x } }\int_0^tL\big(\gamma(\tau),h^{x_0,u_0}(\gamma(\tau),t-\tau),\dot{\gamma}(\tau)\big)d\tau,\label{2-2}
				\end{align}
				where the infimums are taken among the Lipschitz continuous curves $\gamma:[0,t]\rightarrow M$.
				Moreover, the infimums in (\ref{2-1}) and \eqref{2-2} can be achieved.
				If $\gamma_1$ and $\gamma_2$ are curves achieving the infimums \eqref{2-1} and \eqref{2-2} respectively, then $\gamma_1$ and $\gamma_2$ are of class $C^1$.
				Let
				\begin{align*}
					x_1(s)&:=\gamma_1(s),\quad u_1(s):=h_{x_0,u_0}(\gamma_1(s),s),\,\,\,\qquad  p_1(s):=\frac{\partial L}{\partial \dot{x}}(\gamma_1(s),u_1(s),\dot{\gamma}_1(s)),\\
					x_2(s)&:=\gamma_2(s),\quad u_2(s):=h^{x_0,u_0}(\gamma_1(s),t-s),\quad   p_2(s):=\frac{\partial L}{\partial \dot{x}}(\gamma_2(s),u_2(s),\dot{\gamma}_2(s)).
				\end{align*}
				Then $(x_1(s),u_1(s),p_1(s))$ and $(x_2(s),u_2(s),p_2(s))$ satisfy equations \eqref{c} with
				\begin{align*}
					x_1(0)=x_0, \quad x_1(t)=x, \quad \lim_{s\rightarrow 0^+}u_1(s)=u_0,\\
					x_2(0)=x, \quad x_2(t)=x_0, \quad \lim_{s\rightarrow t^-}u_2(s)=u_0.
				\end{align*}
			\end{prop}
			We call $h_{x_0,u_0}(x,t)$ (resp. $h^{x_0,u_0}(x,t)$) a  forward (resp. backward) implicit action function associated with $L$
			and the curves achieving the infimums in (\ref{2-1}) (resp. \eqref{2-2}) minimizers of $h_{x_0,u_0}(x,t)$ (resp. $h^{x_0,u_0}(x,t)$). The relation between forward and backward implicit action functions is as follows: for any given $x_0$, $x\in M$, $u_0$, $u\in\mathbb{R}$ and $t>0$,
			\begin{align}\label{2-r}
				h_{x_0,u_0}(x,t)=u\quad  \text{if and only if}\quad  h^{x,u}(x_0,t)=u_0.
			\end{align}

			See \cite{CCJWY} for another formulation of variational principles from the optimal control point of view.
			This viewpoint is strongly reminiscent of Herglotz' variational principle \cite{H}.
			The following result is a direct consequence of \eqref{2-r}.
			\begin{prop} For any $x$, $y\in M$, any $u\in\R$ and any $t>0$,
				\begin{itemize}
					\item [(1)] $h^{y,h_{x,u}(y,t)}(x,t)=u,$
					\item [(2)] $h_{y,h^{x,u}(y,t)}(x,t)=u.$
				\end{itemize}
			\end{prop}

			\medskip
			
			\noindent $\bullet$ {\bf Implicit action functions}.
			We now collect some basic properties of implicit action functions.

			\begin{prop}\label{pr-af} \
				\begin{itemize}
					\item [(1)] (Monotonicity).
					Given $x_{0} \in M, u_{0}, u_{1}, u_{2} \in \mathbb{R}$, Lagrangians $L_{1}$ and $L_{2}$ satisfying (L1)-(L3),
					\begin{itemize}
						\item [(i)] if $u_{1}<u_{2}$, then $h_{x_{0}, u_{1}}(x, t)<h_{x_{0}, u_{2}}(x, t)$, for all $(x, t) \in M \times(0,+\infty)$;
						\item [(ii)] if $L_{1}<L_{2}$, then $h_{x_{0}, u_{0}}^{L_{1}}(x, t)<h_{x_{0}, u_{0}}^{L_{2}}(x, t)$, for all $(x, t) \in M \times(0,+\infty)$
						where $h_{x_{0}, u_{0}}^{L_{i}}(x, t)$ denotes the forward implicit action function associated with $L_{i}, i=1,2 .$
					\end{itemize}
					\item [(2)] (Lipschitz continuity).
					The function $(x_0,u_0,x,t)\mapsto h_{x_0,u_0}(x,t)$ is  Lipschitz  continuous on $M\times[a,b]\times M\times[c,d]$ for all real numbers $a$, $b$, $c$, $d$
					with $a<b$ and $0<c<d$.
					\item [(3)] (Minimality).
					Given $x_0$, $x\in M$, $u_0\in\mathbb{R}$ and $t>0$, let
					$S^{x,t}_{x_0,u_0}$ be the set of the solutions $(x(s),u(s),p(s))$ of (\ref{c}) on $[0,t]$ with $x(0)=x_0$, $x(t)=x$, $u(0)=u_0$.
					Then
					\[
					h_{x_0,u_0}(x,t)=\inf\{u(t): (x(s),u(s),p(s))\in S^{x,t}_{x_0,u_0}\}, \quad \forall (x,t)\in M\times(0,+\infty).
					\]
					\item [(4)] (Markov property).
					Given $x_0\in M$, $u_0\in\mathbb{R}$,
					\[
					h_{x_0,u_0}(x,t+s)=\inf_{y\in M}h_{y,h_{x_0,u_0}(y,t)}(x,s)
					\]
					for all  $s$, $t>0$ and all $x\in M$. Moreover, the infimum is attained at $y$ if and only if there exists a minimizer $\gamma$ of $h_{x_0,u_0}(x,t+s)$ with $\gamma(t)=y$.
					\item [(5)] (Reversibility).
					Given $x_0$, $x\in M$ and $t>0$, for each $u\in \mathbb{R}$, there exists a unique $u_0\in \mathbb{R}$ such that
					\[
					h_{x_0,u_0}(x,t)=u.
					\]
				\end{itemize}
			\end{prop}

			\begin{prop}\label{pr-af1} \
				\begin{itemize}
					\item [(1)]({\it Monotonicity}).
					Given $x_{0} \in M$ and $u_{1}, u_{2} \in \mathbb{R}$, Lagrangians $L_{1}, L_{2}$ satisfying (L1)-(L3),
					\begin{itemize}
						\item [(i)] if $u_{1}<u_{2}$, then $h^{x_{0}, u_{1}}(x, t)<h^{x_{0}, u_{2}}(x, t)$, for all $(x, t) \in M \times(0,+\infty) ;$
						\item [(ii)] if $L_{1}>L_{2}$, then $h_{L_{1}}^{x_{0}, u_{0}}(x, t)<h_{L_{2}}^{x_{0}, u_{0}}(x, t)$, for all $(x, t) \in M \times(0,+\infty)$, where
						$h_{L_{i}}^{x_{0}, u_{0}}(x, t)$ denotes the backward implicit action function associated with $L_{i}$, $i =1,2$.
					\end{itemize}
					\item [(2)] (Lipschitz continuity).
					The function $(x_0,u_0,x,t)\mapsto h^{x_0,u_0}(x,t)$ is  Lipschitz  continuous on $M\times[a,b]\times M\times[c,d]$ for all real numbers $a$, $b$, $c$, $d$
					with $a<b$ and $0<c<d$.
					
					\item [(3)] (Maximality).
					Given $x_0$, $x\in M$, $u_0\in\mathbb{R}$ and $t>0$, let
					$S_{x,t}^{x_0,u_0}$ be the set of the solutions $(x(s),u(s),p(s))$ of  (\ref{c}) on $[0,t]$ with $x(0)=x$, $x(t)=x_0$, $u(t)=u_0$.
					Then
					\[
					h^{x_0,u_0}(x,t)=\sup\{u(0): (x(s),u(s),p(s))\in S_{x,t}^{x_0,u_0}\}, \quad \forall (x,t)\in M\times(0,+\infty).
					\]
					\item [(4)] (Markov property).
					Given $x_0\in M$, $u_0\in\mathbb{R}$,
					\[
					h^{x_0,u_0}(x,t+s)=\sup_{y\in M}h^{y,h^{x_0,u_0}(y,t)}(x,s)
					\]
					for all  $s$, $t>0$ and all $x\in M$. Moreover, the supremum is attained at $y$ if and only if there exists a minimizer $\gamma$ of $h^{x_0,u_0}(x,t+s)$, such that $\gamma(t)=y$.
					\item [(5)] (Reversibility).
					Given $x_0$, $x\in M$, and $t>0$, for each $u\in \mathbb{R}$, there exists a unique $u_0\in \mathbb{R}$ such that
					\[
					h^{x_0,u_0}(x,t)=u.
					\]
				\end{itemize}
			\end{prop}

			We will use the following result to prove Proposition \ref{iv} below. See the Appendix for the proof.
			\begin{prop}\label{pr2.3}
				Let $(x(t),u(t)):\mathbb{R}\to M\times\mathbb{R}$ be
				a locally Lipschitz curve satisfying
				\begin{align*}
					u(t_2)=h_{x(t_1),u(t_1)}(x(t_2),t_2-t_1)
				\end{align*}
				for all $t_1$, $t_2\in\mathbb{R}$ with $t_1< t_2$. Then $x(t)$	is of class $C^1$. Let $p(t):=\frac{\partial L}{\partial \dot{x}}(x(t),u(t),\dot{x}(t))$. Then $(x(t),u(t),p(t))$ is a solution of  \eqref{c}. Moreover, for each $t_1$, $t_2\in\mathbb{R}$ with $t_1<t_2$, $x(t)\big|_{[t_1,t_2]}$ is a minimizer of $h_{x(t_1),u(t_1)}(x(t_2),t_2-t_1)$.
			\end{prop}
			
			\medskip

			
			\medskip
			
			\noindent $\bullet$ {\bf Solution semigroups}.
			We collect some basic properties of the solution semigroups.
			\begin{prop}\label{pr-sg}
				Let $\varphi$, $\psi\in C(M,\R)$.
				\begin{itemize}
					\item [(1)](Monotonicity). If $\psi<\varphi$, then $T^{\pm}_t\psi< T^{\pm}_t\varphi$, $\forall t\geqslant 0$.
					\item [(2)](Local Lipschitz continuity). The function $(x,t)\mapsto T^{\pm}_t\varphi(x)$ is locally Lipschitz on $M\times (0,+\infty)$.
					\item[(3)]($e^{\lambda t}$-expansiveness). $\|T^{\pm}_t\varphi-T^{\pm}_t\psi\|_\infty\leqslant e^{\lambda t}\cdot\|\varphi-\psi\|_\infty$,  $\forall t\geqslant 0$.
					\item[(4)] (Continuity at the origin). $\lim_{t\rightarrow0^+}T^{\pm}_t\varphi=\varphi$.
					\item[(5)] (Representation formula). For each  $\varphi\in C(M,\R)$,	
					\begin{itemize}
						\item [(i)]	$T^-_t\varphi(x)=\inf_{y\in M}h_{y,\varphi(y)}(x,t)$,\quad  $\forall (x,t)\in M\times(0,+\infty)$;
						\item [(ii)] $T^+_t\varphi(x)=\sup_{y\in M}h^{y,\varphi(y)}(x,t)$,\quad  $\forall (x,t)\in M\times(0,+\infty)$.
					\end{itemize}
					
					\item[(6)] (Semigroup). $\{T^{\pm}_t\}_{t\geqslant 0}$ are one-parameter semigroup of operators. For all $x_0$, $x\in M$, all $u_0\in\mathbb{R}$ and all $s$, $t>0$, 	\begin{itemize}
						\item [(i)] $
						T^-_sh_{x_0,u_0}(x,t)=h_{x_0,u_0}(x,t+s)$,\quad  $T^-_{t+s}\varphi(x)=\inf_{y\in M}h_{y,T^-_s\varphi(y)}(x,t)$;
						\item [(ii)] $
						T^+_sh^{x_0,u_0}(x,t)=h^{x_0,u_0}(x,t+s)$, \quad  $T^+_{t+s}\varphi(x)=\sup_{y\in M}h^{y,T^+_s\varphi(y)}(x,t)$.
					\end{itemize}
				\end{itemize}
			\end{prop}

			
			\medskip
			
			\noindent $\bullet$ {\bf Weak KAM solutions}.
			Following Fathi (see, for instance, \cite{Fat-b}), one can define weak KAM solutions of  equation
			\begin{align}\label{shj}
				H(x,u(x),Du(x))=0
			\end{align}
			as follows.

			\begin{de}\label{bwkam}
				A function $u\in C(M,\R)$ is called a backward weak KAM solution of \eqref{shj} if
				\begin{itemize}
					\item [(1)] for each continuous piecewise $C^1$ curve $\gamma:[t_1,t_2]\rightarrow M$, we have
					\begin{align}\label{do}
						u(\gamma(t_2))-u(\gamma(t_1))\leqslant\int_{t_1}^{t_2}L(\gamma(s),u(\gamma(s)),\dot{\gamma}(s))ds;
					\end{align}
					\item [(2)] for each $x\in M$, there exists a $C^1$ curve $\gamma:(-\infty,0]\rightarrow M$ with $\gamma(0)=x$ such that
					\begin{align}\label{cali1}
						u(x)-u(\gamma(t))=\int^{0}_{t}L(\gamma(s),u(\gamma(s)),\dot{\gamma}(s))ds, \quad \forall t<0.
					\end{align}
				\end{itemize}
				
				Similarly, 	a function $v\in C(M,\R)$ is called a forward weak KAM solution of  \eqref{shj} if it satisfies (1) and
				for each $x\in M$, there exists a $C^1$ curve $\gamma:[0,+\infty)\rightarrow M$ with $\gamma(0)=x$ such that
				\begin{align}\label{cali2}
					v(\gamma(t))-v(x)=\int_{0}^{t}L(\gamma(s),v(\gamma(s)),\dot{\gamma}(s))ds,\quad \forall t>0.
				\end{align}
				We say that $u$ in \eqref{do} is a dominated function by $L$, denoted by $u\prec L$.
				We call curves satisfying \eqref{cali1} (resp. \eqref{cali2}) ,  $(u,L,0)$-calibrated curves (resp. $(v,L,0)$-calibrated curves).
			\end{de}
			
			Under assumptions (H1)-(H3), backward weak KAM solutions are viscosity solutions.
			
			\begin{prop}\label{pr-fix}\
				\begin{itemize}
					\item
					[(1)] $u\in\mathcal{S}_-$ if and only if $T^-_tu=u$ for all $t\geqslant 0$.
					\item [(2)] $v\in\mathcal{S}_+$ if and only if $T^+_tv=v$ for all $t\geqslant 0$.
				\end{itemize}
			\end{prop}

			The following result will be useful for the proof of Proposition \ref{pr1}.  We give the proof in the Appendix 6.2, since it is quite lengthy.
			\begin{prop}\label{iv}
				Let $u\in\mathcal{S}_-$. Given any $x\in M$,
				if $\gamma:(-\infty,0]\rightarrow M$ is a  $(u,L,0)$-calibrated curve  with $\gamma(0)=x$, then $\big(\gamma(t),u(\gamma(t)),p(t)\big)$ satisfies equations \eqref{c} on $(-\infty,0)$, where $p(t)=\frac{\partial L}{\partial \dot{x}}(\gamma(t),u(\gamma(t)),\dot{\gamma}(t))$.
				Moreover, we have
				\[
				\big(\gamma(t+s),u(\gamma(t+s)),Du(\gamma(t+s)\big)=\Phi_{s}\big(\gamma(t),u(\gamma(t)),Du(\gamma(t)\big), \quad\forall t, \ s<0,
				\]
				and
				\[
				H\big(\gamma(t),u(\gamma(t)),\frac{\partial L}{\partial \dot{x}}(\gamma(t),u(\gamma(t)),\dot{\gamma}(t))\big)=0,\quad\forall t< 0.
				\]
			\end{prop}
			
			\medskip
			
			\begin{re}\label{iv1}
				A similar result holds true for  $v\in \mathcal{S}_+$:  let $v\in\mathcal{S}_+$. Given any $x\in M$,
				if $\gamma:[0,+\infty)\rightarrow M$ is a  $(v,L,0)$-calibrated curve  with $\gamma(0)=x$, then $\big(\gamma(t),v(\gamma(t)),p(t)\big)$ satisfies equations \eqref{c} on $(0,+\infty)$, where $p(t)=\frac{\partial L}{\partial \dot{x}}(\gamma(t),v(\gamma(t)),\dot{\gamma}(t))$.
				Moreover, we have
				\[
				\big(\gamma(t+s),v(\gamma(t+s)),Dv(\gamma(t+s)\big)=\Phi_{s}\big(\gamma(t),v(\gamma(t)),Dv(\gamma(t)\big), \quad\forall t, \ s>0.
				\]
				Since the proof of the above result is quite similar to the one of Proposition \ref{iv}, we omit it.
			\end{re}
			
			\medskip
			Let $u\in \mathcal{S}_-$ and $v\in \mathcal{S}_+$. In view of  Lemma \ref{ulipp} in the Appendix 6.2, both $u$  and $v$ are Lipschitz continuous.

			\subsection{More on solution semigroups}

			\begin{prop}\label{pr-v}
				Let $\varphi\in C(M,\R)$. If the function $(x,t)\mapsto T^-_t\varphi(x)$ is bounded on $M\times[0,+\infty)$, then $\varphi_\infty(x):=\liminf_{t\to+\infty}T^-_t\varphi(x)$ is a viscosity solution of \eqref{shj}.
			\end{prop}

			\begin{proof}
				Let $K_1$ be a positive constant such that
				\begin{align}\label{3-10}
					|T^-_t\varphi(x)|\leqslant K_1,\quad \forall
					(x,t)\in M\times[0,+\infty).
				\end{align}
				Recall Proposition \ref{pr-af}(2), i.e., the function $(x_0,u_0,x,t)\mapsto h_{x_0,u_0}(x,t)$ is Lipschitz on $M\times [a,b]\times M\times [c,d]$ for all real numbers $a$, $b$, $c$, $d$
				with $a<b$ and $0<c<d$.

				First we show that $\{T^-_t\varphi(x)\}_{t>1}$ is equi-Lipschitz on $M$.
				Denote by $l_1>0$ a Lipschitz constant of the function $(x_0,u_0,x)\mapsto h_{x_0,u_0}(x,1)$ on $M\times [-K_1,K_1]\times M$. From Proposition \ref{pr-sg} (6)(i), we have
				\[
				|T^-_t\varphi(x)-T^-_t\varphi(y)|\leqslant \sup_{z\in M}|h_{z,T^-_{t-1}\varphi(z)}(x,1)-h_{z,T^-_{t-1}\varphi(z)}(y,1)|
				\]
				for all $x$, $y\in M$, and all $t>1$. In view of \eqref{3-10}, the above inequality implies that
				\[
				|T^-_t\varphi(x)-T^-_t\varphi(y)|\leqslant l_1 \cdot d(x,y).
				\]

				Then let $\varphi_\infty(x):=\liminf_{t\to+\infty}T^-_t\varphi(x)$. We show that $\varphi_\infty$ is a fixed point of $\{T^-_t\}_{t\geqslant 0}$.
				Since $\{T^-_t\varphi(x)\}_{t>1}$ is equi-Lipschitz on $M$, it is easy to see that
				\begin{align}\label{3-11}
					\lim_{t\to+\infty}\inf_{s\geqslant t}T^-_s\varphi(x)=\varphi_\infty(x)\quad \text{uniformly on}\  x\in M.	
				\end{align}
				Note that
				\begin{align*}
					\va_\infty(x)=\liminf_{t\to+\infty}T^-_s\circ T^-_t\va(x),\quad \forall s\geqslant 0.
				\end{align*}
				By the definition of liminf, we have
				\begin{align*}
					\va_\infty(x)&=\lim_{m\to+\infty}\lim_{n\to+\infty}\min_{m\leqslant t\leqslant n}T^-_s\circ T^-_t\va(x)\\
					&=\lim_{m\to+\infty}\lim_{n\to+\infty}T^-_s( \min_{m\leqslant t\leqslant n}T^-_t\va)(x)\\
					&=T^-_s( \lim_{m\to+\infty}\lim_{n\to+\infty}\min_{m\leqslant t\leqslant n}T^-_t\va)(x)\\
					&=T^-_s\va_\infty(x).
				\end{align*}	
			\end{proof}

			\begin{prop}\label{pr-sg1}
				Let $\varphi\in C(M,\R)$. Then
				\begin{itemize}
					\item[(1)] $T^-_t\circ T^+_t\va\geqslant \va$, \quad $\forall t>0$;
					\item[(2)] $T^+_t\circ T^-_t\va\leqslant \va$, \quad $\forall t>0$.
				\end{itemize}
			\end{prop}
			
			\begin{proof}
				For any $x\in M$ and $t>0$, we have that	
				\[
				T^-_t\circ T^+_t\va(x)=\inf_{y\in M}h_{y,T^+_t\va(y)}(x,t)\geqslant \inf_{y\in M}h_{y,h^{x,\va(x)}(y,t)}(x,t)=\va(x),
				\]
				and
				\[
				T^+_t\circ T^-_t\va(x)=\sup_{y\in M}h^{y,T^-_t\va(y)}(x,t)\leqslant \sup_{y\in M}h^{y,h_{x,\va(x)}(y,t)}(x,t)=\va(x).
				\]
			\end{proof}

			\medskip
			
			\begin{prop}\label{pr1}\
				\begin{itemize}
					\item [(1)] For each $u\in\mathcal{S}_-$, the uniform limit $\lim_{t\to+\infty}T^+_tu=:v$ exists and $v\in\mathcal{S}_+$.
					\item [(2)] For each $v\in\mathcal{S}_+$, the uniform limit $\lim_{t\to+\infty}T^-_tv=:u$ exists and $u\in\mathcal{S}_-$.
				\end{itemize}
			\end{prop}
			
			\medskip	
			
			Before showing Proposition \ref{pr1}, we need to prove some preliminary results.
			\begin{prop}\label{pr3.3}
				Let $u$, $v\in C(M,\R)$ and let $t\geqslant 0$. Then $v\leqslant T^-_tu$ if and only if $T_t^+v\leqslant u$.
			\end{prop}
			\begin{proof}
				If $v\leqslant T^-_tu$ for some $t\geqslant 0$, we will show $T_t^+v\leqslant u$.
				It is clear that $T_0^+v(x)=v(x)$. Fix $(x,t)\in M\times (0,+\infty)$. By Proposition \ref{pr-sg}(5)(ii), we have
				\[
				T^{+}_tv(x)=\sup_{y\in M}h^{y,v(y)}(x,t).
				\]
				It suffices to prove that $h^{y,v(y)}(x,t)\leqslant u(x)$ for all $y\in M$.
				Let $\varphi(y):=h^{y,v(y)}(x,t)$ for all $y\in M$. Then by \eqref{2-r}, we have that  $v(y)=h_{x,\varphi(y)}(y,t)$ for all $y\in M$. In view of $v\leqslant T^-_tu$ and Proposition \ref{pr-sg}(5)(i), for each $y\in M$, we get that
				\[
				v(y)\leqslant T_t^-u(y)=\inf_{z\in M}h_{z,u(z)}(y,t),
				\]
				which implies $v(y)\leqslant h_{x,u(x)}(y,t)$,  that is, $h_{x,\varphi(y)}(y,t)\leqslant h_{x,u(x)}(y,t)$. By Proposition \ref{pr-af}(1)(i), we have $\varphi(y)\leqslant u(x)$ for each $y\in M$.

				The converse implication can be proved in a similar manner.
			\end{proof}
			
			\medskip
			%

			The following result is a direct consequence of the above proposition.
			\begin{cor}
				Solution semigroups $\{T^{-}_t\}_{t\geqslant 0}$ and $\{T^{+}_t\}_{t\geqslant 0}$ preserve the set of viscosity subsolutions of equation \eqref{shj}.
			\end{cor}

			\begin{prop}\label{pr3.4}
				Let $u\in\mathcal{S}_-$. For each $x\in M$, let $\gamma:(-\infty,0]\to M$  be a $(u,L,0)$-calibrated curve with $\gamma(0)=x$. Then,  $T_t^+u(\gamma(-t))=u(\gamma(-t))$ for all $t\geqslant 0$.
			\end{prop}

			\begin{proof}
				For any given $t>0$, let $z=\gamma(-t)$ and $u_t=u(z)$. By Proposition \ref{pr3.3}, we only need to prove $T_t^+u(z)\geqslant u_t$. By Proposition \ref{pr-sg}(5)(ii), we have
				\begin{align}\label{3-25}
					T_t^+u(z)=\sup_{y\in M}h^{y,u(y)}(z,t)\geqslant h^{x,u(x)}(z,t).
				\end{align}
				So, it suffices to show $h^{x,u(x)}(z,t)\geqslant u_t$.
				By Proposition \ref{iv}, $\big(\gamma(s),u(\gamma(s)),p(s)\big)$ satisfies equations \eqref{c} on $(-\infty,0)$, where $p(s)=\frac{\partial L}{\partial \dot{x}}(\gamma(s),u(\gamma(s)),\dot{\gamma}(s))$. Let ${\bf{u}}(s):=u(\gamma(s-t))$ for $s\in[0,t]$. Then ${\bf{u}}(0)=u(\gamma(-t))=u_t$ and ${\bf{u}}(t)=u(\gamma(0))=u(x)$.  By Proposition \ref{pr-af1}(3), it is clear that
				\[
				h^{x,{\bf{u}}(t)}(z,t)\geqslant u_t.
				\]
				This completes the proof.

			\end{proof}

			\begin{cor}\label{ueel}
				Let $u\in\mathcal{S}_-$.
				The family of functions $\{T_t^+u\}_{t>2}$ is uniformly bounded and equi-Lipschitz on $M$.
			\end{cor}
			
			\begin{proof}
				In order to prove the corollary, we proceed in two steps.
				
				\medskip
				
				\noindent {\bf Step 1}: we first prove the uniform  boundedness of $\{T_t^+u\}_{t>2}$. By Proposition \ref{pr3.3} and the compactness of $M$, the function $(x,t)\mapsto T_t^+u(x)$ is bounded from above on $M\times [0,+\infty)$.

				On the other hand, since $u\in\mathcal{S}_-$, then for any given $y\in M$,
				there is a $(u,L,0)$-calibrated curve $\gamma:(-\infty,0]\to M$ with $\gamma(0)=y$. By Proposition \ref{pr3.4},
				$T_t^+u(\gamma(-t))=u(\gamma(-t))$ for all $t>0$. For any $t>1$ and any $x\in M$, from Proposition \ref{pr-sg}(6), (5)(ii),  we deduce that
				\begin{align*}
					T_t^+u(x)=T_1^+\circ T_{t-1}^+u(x)
					& =\sup_{z\in M}h^{z,T_{t-1}^+u(z)}(x,1)\\
					&\geqslant  h^{\gamma(-t),T_{t-1}^+u(\gamma(-t))}(x,1)\\
					&= h^{\gamma(-t),u(\gamma(-t))}(x,1).
				\end{align*}
				By Proposition \ref{pr-af1}(2), the function $h^{\cdot,\cdot}(\cdot,1)$ is bounded on $M\times[-\|u\|_\infty,\|u\|_\infty]\times M$.
				Thus, the function $(x,t)\mapsto T_t^+u(x)$ is bounded form below on $M\times(1,+\infty)$.
				
				\medskip
				
				\noindent {\bf Step 2}:  we show the equi-Lipschitz property of   $\{T_t^+u\}_{t>2}$. Denote by $K_2>0$ a constant such that $\|T_t^+u\|_\infty\leq K_2$ for all $t>1$.
				In view of Proposition \ref{pr-sg}(6)(ii), for any $x$, $y\in M$, we get that
				\begin{align*}
					|T_t^+u(x)-T_t^+u(y)|&=|\sup_{z\in M}h^{z,T_{t-1}^+u(z)}(x,1)-\sup_{z\in M}h^{z,T_{t-1}^+u(z)}(y,1)|\\
					&\leqslant \sup_{z\in M}|h^{z,T_{t-1}^+u(z)}(x,1)-h^{z,T_{t-1}^+u(z)}(y,1)|.
				\end{align*}
				By Proposition \ref{pr-af1}(2), the function $h^{\cdot,\cdot}(\cdot,1)$ is  Lipschitz on $M\times [-K_2,K_2]\times M$ with a Lipschitz constant $\kappa_1>0$, and thus we get that
				\[|T_t^+u(x)-T_t^+u(y)|\leqslant \kappa_1 d(x,y), \quad \forall t>2.\]
			\end{proof}

			\begin{proof}[Proof of Proposition \ref{pr1}]
				(1) By Proposition  \ref{pr3.3} and Corollary  \ref{ueel}, the uniform limit $\lim_{t\rightarrow+\infty}T_t^+u$ exists. Define
				\[v:=\lim_{t\rightarrow+\infty}T_t^+u.\]
				It follows from Proposition \ref{pr-sg}(3) that, for any given $t\geq 0$, we get that
				\[
				\|T_{t+s}^+u-T_t^+v\|_\infty\leqslant  e^{\lambda t}\|T_s^+u-v\|_\infty, \quad \forall s>0.
				\]
				Letting $s\rightarrow +\infty$, we have
				\[T_t^+v(x)=v(x),\quad \forall x\in M.\]
				By Proposition \ref{pr-fix}(2), we deduce that $v\in \mathcal{S}_+$. The proof of Proposition \ref{pr1}(1) is complete.
				
				\medskip
				
				(2) Now we turn to the proof of Proposition \ref{pr1}(2). Since the proof of (2) is quite similar to the one of (1), we only sketch the strategy of the proof.
				
				Let $v\in\mathcal{S}_+$. By similar arguments used in the proofs of Propositions \ref{pr3.3}, \ref{pr3.4} and Corollary \ref{ueel}, one can show that
				\begin{itemize}
					\item [(a)] $T_t^-v\geqslant v$ for all $t\geqslant 0$.
					\item [(b)] For each $x\in M$, let $\gamma:[0,+\infty)\to M$  be a $(v,L,0)$-calibrated curve with $\gamma(0)=x$. Then,  $T_t^-v(\gamma(t))=v(\gamma(t))$ for all $t\geqslant 0$.
					\item [(c)] The family of functions $\{T_t^-v\}_{t>2}$ is uniformly bounded and equi-Lipschitz on $M$.
				\end{itemize}
				Using the above three facts, we deduce that the uniform limit
				$\lim_{t\to+\infty}T_t^-v=:u$ exists and that $u\in\mathcal{S}_-$.

				The proof of Proposition \ref{pr1} is now complete.
			\end{proof}

			%
			%


			\medskip
			
			\section{Existence of solutions of the generalized ergodic problem}
			
			\setcounter{equation}{0}
			
			We prove Theorem \ref{thA} in this section. Theorem \ref{thA} is a direct consequence of Proposition \ref{thG} and Proposition \ref{thH} below.
			
			%
			
			\medskip

			\subsection{Necessary and sufficient conditions for the existence  I}
			
			This part is devoted to the following result.	
			\begin{prop}\label{thG}
				Let $c\in\R$. The following  statements are equivalent.
				\begin{itemize}
					\item [(1)] Equation \eqref{shjc} has viscosity solutions;
					\item [(2)] There exist $\varphi$, $\psi\in C(M,\R)$ and $t_1$, $t_2\in \R_+$ such that $T^{+,c}_{t_1}\varphi\leqslant \varphi$, $T^{+,c}_{t_2}\psi\geqslant \psi$.
				\end{itemize}
			\end{prop}	
			
			\medskip

			In the rest of subsection 3.1, without any loss of generality we assume that $c=0$. Consider the contact Hamilton-Jacobi equation
				\begin{align}\label{s0}
				H(x,u(x),Du(x))=0.
			\end{align}
			
			\begin{lem}\label{lem5-1}
				Let $\va\in C(M,\R)$. If $\va\leqslant T^{-}_t\va$  (resp. $\va\geqslant T^{-}_t\va$) for all $t\geqslant 0$, then either
				\[
				\lim_{t\to+\infty}T^{-}_t\va(x)=+\infty \quad  (\text{resp}. \ \lim_{t\to+\infty}T^{-}_t\va(x)=-\infty)
				\]
				uniformly on  $x\in M$, or
				\[
				\lim_{t\to+\infty}T^{-}_t\va(x)=\va_\infty(x)
				\]
				uniformly on $x\in M$, where $\va_\infty(x)$ is  a viscosity solution of \eqref{s0}.
			\end{lem}

			\begin{proof} We divide the proof in two steps.
				
				\medskip
				\noindent {\bf Step 1}:
				First, we consider the case $\varphi\leqslant T^{-}_t\varphi$ for all $t\geqslant 0$.
				In view of $\varphi\leqslant T^{-}_t\varphi$,
				if $\lim_{t\to+\infty}T^{-}_t\varphi(x_0)=B<+\infty$ for some $x_0\in M$, then for each $x\in M$, we have
				\begin{align*}
					T^{-}_t\varphi(x)=(T^{-}_1\circ T^{-}_{t-1})\varphi(x)=\inf_{y\in M}h_{y,T^{-}_{t-1}\varphi(y)}(x,1)\leqslant h_{x_0,T^{-}_{t-1}\varphi(x_0)}(x,1)\leqslant h_{x_0,B}(x,1)<+\infty.
				\end{align*}
				Thus, we deduce that for any $t>1$,
				\begin{align}\label{5-1}
					-\|\va\|_\infty\leqslant T^{-}_t\varphi(x)\leqslant\max_{y,y'\in M}h_{y,B}(y',1),\quad \forall x\in M.
				\end{align}
				So, there are two possibilities: (i) $\lim_{t\to+\infty}T^{-}_t\varphi(x)=\varphi_\infty(x)$ for all $x\in M$. In view of \eqref{5-1}, by Proposition
				\ref{pr-v}, $\lim_{t\to+\infty}T^{-}_t\varphi(x)=\varphi_\infty(x)$ uniformly on $x\in M$, and $\va_\infty$ is a viscosity solution of \eqref{s0}. (ii) $\lim_{t\to+\infty}T^{-}_t\varphi(x)=+\infty$ for all $x\in M$. Next, we prove
				$\lim_{t\to+\infty}T^{-}_t\varphi(x)=+\infty$ uniformly on $x\in M$. Suppose not. Then there are $K_0>0$, $\{t_n\}\nearrow+\infty$ and $x_n\in M$, such that
				\[
				T^{-}_{t_n}\va(x_n)\leqslant K_0.
				\]
				Then for any $n\in \mathbb{N}$, any $x\in M$
				\[
				T^{-}_{t_n+1}\va(x)\leqslant h_{x_n,T^{-}_{t_n}\va(x_n)}(x,1)\leqslant h_{x_n,K_0}(x,1)\leqslant \max_{y',y''\in M}h_{y',K_0}(y'',1)<+\infty,
				\]
				which contradicts $\lim_{t\to+\infty}T^{-}_t\varphi(x)=+\infty$ for all $x\in M$.
				
				\medskip
				
				\noindent {\bf Step 2}: 	
				Second, we deal with the case $\varphi\geqslant T^{-}_t\varphi$ for all $t\geqslant 0$. If $\lim_{t\to+\infty}T^-_t\varphi(x_0)=B'>-\infty$ for some $x_0\in M$, then for any $y\in M$, we have that
				\[
				v_t:=h_{y,T^{-}_t\va(y)}(x_0,1)\geqslant T^{-}_{t+1}\va(x_0)\geqslant B'.
				\]
				So, we get that
				\[
				h^{x_0,v_t}(y,1)=T^{-}_t\va(y).
				\]
				Thus, one can deduce that for any $t>0$
				\begin{align}\label{5-21}
					\|\va\|_\infty\geqslant\va(y)\geqslant T^{-}_t\va(y)\geqslant h^{x_0,B'}(y,1)\geqslant -\|h^{x_0,B'}(\cdot,1)\|_\infty,\quad \forall y\in M.
				\end{align}
				So, there are two possibilities: (i) $\lim_{t\to+\infty}T^{-}_t\varphi(x)=\varphi'_\infty(x)$ for all $x\in M$. In view of \eqref{5-21}, by Proposition
				\ref{pr-v}, $\lim_{t\to+\infty}T^{-}_t\varphi(x)=\varphi'_\infty(x)$ uniformly on $x\in M$, and $\va'_\infty$ is a viscosity solution of \eqref{s0}. (ii) $\lim_{t\to+\infty}T^{-}_t\varphi(x)=-\infty$ for all $x\in M$. Next, we prove
				$\lim_{t\to+\infty}T^{-}_t\varphi(x)=-\infty$ uniformly on $x\in M$. Suppose not. Then there are $K'_0<0$, $\{t'_n\}\nearrow+\infty$ and $x'_n\in M$, such that
				\[
				T^{-}_{t'_n}\va(x'_n)\geqslant K'_0.
				\]
				Given any $y\in M$, let
				\[
				v'_n:=h_{y,T^{-}_{t'_n-1}\va(y)}(x'_n,1)\geqslant T^{-}_{t'_n}\va(x'_n)\geqslant K'_0.
				\]
				By \eqref{2-r},
				\[
				h^{x'_n,v'_n}(y,1)=T^{-}_{t'_n-1}\va(y).
				\]
				Thus,
				\[
				T^{-}_{t'_n-1}\va(y)\geqslant h^{x'_n,K'_0}(y,1)\geqslant \min_{z,z'\in M}h^{z,K'_0}(z',1)>-\infty,
				\]
				which contradicts $\lim_{t\to+\infty}T^{-}_t\varphi(y)=-\infty$.

				The proof is complete.
			\end{proof}

			The following  result is a direct consequence of Lemma \ref{lem5-1}. We omit the proof.
			
			\begin{cor}\label{co5-1} Let $\va\in C(M,\R)$.
				If $\va\leqslant T^{-}_{t_0}\va$ (resp. $\va\geqslant T^{-}_{t_0}\va$) for some $t_0> 0$, then either  \[\lim_{n\to+\infty}T^{-}_{nt_0}\va(x)=+\infty\quad  (\text{resp.}\  \lim_{n\to+\infty}T^{-}_{nt_0}\va(x)=-\infty)
				\] uniformly on $x\in M$, or $\lim_{n\to+\infty}T^{-}_{nt_0}\va(x)=\va_\infty(x)$ uniformly on $x\in M$, where $\va_\infty(x)\in\mathrm{Lip}(M)$.
			\end{cor}


			\medskip

			\begin{proof}[Proof of Proposition \ref{thG}]
				In view of Proposition \ref{pr-sg1}, notice that $T^-_t\va\geqslant \va$ if and only if $T^+_t\va\leqslant \va$.

				If equation \eqref{s0} has a viscosity solution $\va$, then by Proposition \ref{pr1} one can deduce that $\psi:=\lim_{t\to+\infty}T^+_t\va$ is a forward weak KAM solution. Thus, $T^-_t\va=\va$ and $T^+_t\psi=\psi$ for all $t\geqslant 0$, i.e., item (2) in Proposition \ref{thG} holds true.

				Let $\va\in C(M,\R)$ and $t_1>0$ be such that $T^-_{t_1}\va\geqslant\va$. Let $\psi\in C(M,\R)$ and $t_2>0$ be such that $T^+_{t_2}\psi\geqslant\psi$.
				We aim to prove that equation \eqref{s0} has viscosity solutions.

				Since $T^-_{t_1}\va\geqslant\va$, then by Corollary \ref{co5-1},  either  $\lim_{n\to+\infty}T^{-}_{nt_1}\va(x)=+\infty$ uniformly on $x\in M$, or $\lim_{n\to+\infty}T^{-}_{nt_1}\va(x)=\va_\infty(x)$ uniformly on $x\in M$, where $\va_\infty(x)\in\mathrm{Lip}(M)$.

				\medskip
				
				\noindent {\bf Case 1}:	If $\lim_{n\to+\infty}T^{-}_{nt_1}\va(x)=\va_\infty(x)$ uniformly on $x\in M$, then for any $s\in[0,t_1]$, $\lim_{n\to\infty}T^-_{nt_1+s}\va(x)=T^-_s\va_\infty(x)$.
				Hence, $T^{-}_{t}\va(x)$ is bounded on $M\times[0,+\infty)$. Thus, by Proposition \ref{pr-v},
				\[
				\va'(x):=\liminf_{t\to+\infty}T^{-}_{t}\va(x)
				\]
				is a viscosity solution of equation \eqref{s0}.

				\medskip
				
				\noindent {\bf Case 2}:	If $\lim_{n\to+\infty}T^{-}_{nt_1}\va(x)=+\infty$ uniformly on $x\in M$, then there is $n_1\in\mathbb{N}$ such that $T^-_{n_1t_1}\va>\psi$, and $T^-_{n_1t_1}\va>\va$. Choose $k_1$, $k_2\in\mathbb{N}$ such that
				\[
				s_0:=\frac{k_1}{k_2}t_2-n_1t_1>0
				\]
				small enough with
				\[
				T^-_{n_1t_1+s_0}\va>\psi,\quad T^-_{n_1t_1+s_0}\va>\va.
				\]
				Let $t_0:=k_2(n_1t_1+s_0)=k_1t_2$. Then
				\[
				T^-_{t_0}\va>\va,\quad T^-_{t_0}\va>\psi,\quad T^+_{t_0}\psi\geqslant\psi.
				\]
				Then by Proposition \ref{pr3.3},
				\[
				T^+_{t_0}\va\leqslant\va.
				\]
				Let $\va'=T^-_{t_0}\va$. Then by Proposition \ref{pr-sg1}, we have that
				\[
				T^+_{t_0}\va'=T^+_{t_0}\circ T^-_{t_0}\va\leqslant \va,
				\]
				and
				\[
				T^+_{t_0}\va'\geqslant T^+_{t_0}\psi\geqslant\psi.
				\]
				Therefore, we get that
				\[
				\psi\leqslant T^+_{nt_0}\psi\leqslant T^+_{nt_0}\va'\leqslant \va.
				\]
				So, $\{T^+_{nt_0}\psi\}_{n\in \mathbb{N}}$ is bounded, and thus the uniform limit
				\[
				\lim_{n\to\infty}T^+_{nt_0}\psi=:\psi_\infty
				\]
				exists.
				And for any $s\in[0,t_0]$,
				\[
				\lim_{n\to\infty}T^+_{nt_0+s}\psi(x)=:T^+_{s}\psi_\infty(x),\quad x\in M.
				\]
				It follows that the function $(x,t)\mapsto T^+_{t}\psi(x)$ is bounded on $M\times[0,+\infty)$. Let
				\[
				\psi'_\infty(x):=\limsup_{t\to+\infty}T^+_{t}\psi(x).
				\]
				We assert that $\psi'_\infty$ is a forward weak KAM  solution of equation \eqref{s0}. If the assertion is true, then by Proposition \ref{pr1}, one can deduce that $\mathcal{S}_-\ne\emptyset$.

				So, it suffices to show the assertion. Let $K_3>0$ be such that $|T^+_{t}\psi(x)|\leqslant K_3$ for all $(x,t)\in M\times[0,+\infty)$.
				First we show that $\{T^+_t\psi(x)\}_{t>1}$ is equi-Lipschitz on $M$.
				Denote by $\kappa_2>0$ a Lipschitz constant of the function $(x_0,u_0,x)\mapsto h^{x_0,u_0}(x,1)$ on $M\times [-K_3,K_3]\times M$. From Proposition \ref{pr-sg} (6)(ii), we have
				\[
				|T^+_t\psi(x)-T^+_t\psi(y)|\leqslant \sup_{z\in M}|h^{z,T^+_{t-1}\psi(z)}(x,1)-h^{z,T^+_{t-1}\psi(z)}(y,1)|
				\]
				for all $x$, $y\in M$, and all $t>1$. The above inequality implies that
				\[
				|T^+_t\psi(x)-T^+_t\psi(y)|\leqslant \kappa_2 \cdot d(x,y).
				\]

				Next we show that $\psi'_\infty$ is a fixed point of $\{T^+_t\}_{t\geqslant 0}$.
				Since $\{T^+_t\psi(x)\}_{t>1}$ is equi-Lipschitz on $M$, it is easy to see that
				\begin{align*}\label{3-30}
					\lim_{t\to+\infty}\sup_{s\geqslant t}T^+_s\psi(x)=\psi'_\infty(x)\quad \text{uniformly on}\  x\in M.	
				\end{align*}
				
				For any $s\geqslant 0$, notice that
				\begin{align*}
					\psi'_\infty(x)&=\limsup_{t\to+\infty}T^+_t\psi(x)=\limsup_{t\to+\infty}T^+_s\circ T^+_t\psi(x)\\
					&=\lim_{m\to+\infty}\lim_{n\to+\infty}\max_{m\leqslant t\leqslant n}T^+_s\circ T^+_t\psi(x)\\
					&=\lim_{m\to+\infty}\lim_{n\to+\infty}\max_{m\leqslant t\leqslant n}\sup_{y\in M}h^{y,T^+_t\psi(y)}(x,s)\\
					&=\lim_{m\to+\infty}\lim_{n\to+\infty}\sup_{y\in M}h^{y,\max_{m\leqslant t\leqslant n}T^+_t\psi(y)}(x,s)\\
					&=\lim_{m\to+\infty}\lim_{n\to+\infty}T^+_s(\max_{m\leqslant t\leqslant n}T^+_t\psi)(x)\\
					&=T^+_s(\lim_{m\to+\infty}\lim_{n\to+\infty}\max_{m\leqslant t\leqslant n}T^+_t\psi)(x)\\
					&=T^+_s\psi'_\infty(x).
				\end{align*}	
				The proof is complete.


			\end{proof}

			\subsection{A key proposition}
			In order to complete the proof of Theorem \ref{thA}, we need to prove a technical proposition here which also  provides certain information of the long time behavior of the viscosity solution of the Cauchy problem
		\begin{equation}\label{96-1}
			w_{t}(x,t)+H(x, w(x,t), Dw(x,t))=c
		\end{equation}
			 with $w(x,0)=\va(x)$.
			Moreover,  we will use this proposition again in the proof of Theorem \ref{thC}.
			Recall that the function $(x,t)\mapsto T^{-,c}_t\va(x)$ is the unique viscosity solution of the Cauchy problem.
			\begin{prop}\label{thD} Let $\va\in C(M,\R)$ and let $c\in\R$.
				\begin{itemize}
					\item [(1)] if there is $t_0>0$ such that $T^{-,c}_{t_0}\va\geqslant \va$ (resp. $T^{-,c}_{t_0}\va\leqslant \va$), then for any $s\in[0,t_0]$, $\lim_{n\to+\infty}T^{-,c}_{nt_0+s}\va(x)=+\infty$ (resp. $\lim_{n\to+\infty}T^{-,c}_{nt_0+s}\va(x)=-\infty$) uniformly on $x\in M$, or $\lim_{n\to+\infty}T^{-,c}_{nt_0+s}\va(x)=u(x,s)$ uniformly in $(x,s)\in M\times [0,t_0]$, where $u(x,s)$ is a viscosity solution of equation \eqref{96-1} which is $t_0$-periodic in time.
					\item [(2)] if there is $t_0>0$ such that $T^{-,c}_{t_0}\va> \va$ (resp. $T^{-,c}_{t_0}\va<\va$), then $\lim_{t\to+\infty}T^{-,c}_{t}\va(x)=+\infty$ (resp. $\lim_{t\to+\infty}T^{-,c}_{t}\va(x)=-\infty$) uniformly on $x\in M$, or $\lim_{t\to+\infty}T^{-,c}_{t}\va(x)=\va_\infty(x)$ uniformly on $x\in M$, where $\va_\infty$ is a viscosity solution of equation \eqref{shjc}.
					\item [(3)] if for any $t>0$,
					\[
					\Big\{x\in M: T^{-,c}_t\va(x)=\va(x)\Big\}\neq \emptyset,
					\]
					then $|T^{-,c}_t\va(x)|\leqslant K_\va$ for all $(x,t)\in M\times[0,+\infty)$ and for some constant $K_\va>0$ depending on $\va$.
				\end{itemize}
			\end{prop}
			\medskip

			\begin{proof}[Proof of Proposition \ref{thD}]
				We will split the proof into three steps.
				
				\medskip
				
				\noindent {\bf Step 1}:	
				if there is $t_0>0$ such that $\va\leqslant T^{-,c}_{t_0}\va$, then by Corollary \ref{co5-1}, either  $\lim_{n\to+\infty}T^{-,c}_{nt_0}\va(x)=+\infty$  uniformly on $x\in M$, or,  $\lim_{n\to+\infty}T^{-,c}_{nt_0}\va(x)=\va_\infty(x)$ uniformly on $x\in M$, where $\va_\infty(x)$ is a Lipschitz continuous function on $M$.

				\noindent {\it Case (i) }:	If $\lim_{n\to+\infty}T^{-,c}_{nt_0}\va(x)=+\infty$ uniformly on $x\in M$, then we assert that for any $s\in[0,t_0]$,
				\[
				\lim_{n\to+\infty}T^{-,c}_{nt_0+s}\va(x)=+\infty
				\]
				for all $x\in M$. Note that if there are $s_0\in[0,t_0]$ and $x_0\in M$ such that
				\[
				\lim_{n\to\infty}T^{-,c}_{nt_0+s_0}\va(x_0)=A''<+\infty,
				\]
				then for all $x\in M$,
				\[
				T^{-,c}_{nt_0+s_0}\va(x)\leqslant h^c_{x_0,A''}(x,t_0)\leqslant\max_{y,y''\in M}h^c_{y,A''}(y'',t_0)<+\infty.
				\]
				So, in order to show the above assertion, we can assume  by contradiction that there is $s_0\in[0,t_0]$ such that
				\[
				\lim_{n\to\infty}T^{-,c}_{nt_0+s_0}\va(x)=\va^{s_0}_\infty(x),\quad \forall x\in M,
				\]
				where $\va^{s_0}_\infty(x)$ is a function defined on $M$.
				It is clear that $\{|T^{-,c}_{nt_0+s_0}\va(x)|\}_n$ is bounded by a constant $K>0$. And thus, by similar arguments used in the proof of Proposition \ref{pr-v}, $\{|T^{-,c}_{nt_0+s_0}\va(x)|\}_n$ is equi-Lipschitz. Therefore,
				\[
				\lim_{n\to\infty}T^{-,c}_{nt_0+s_0}\va(x)=\va^{s_0}_\infty(x),
				\]
				uniformly on $x\in M$, and $\va^{s_0}_\infty\in\mathrm{Lip}(M)$. Note that
				\[
				\|T^{-,c}_{t_0-s_0}\circ T^{-,c}_{nt_0+s_0}\va-T^{-,c}_{t_0-s_0}\va^{s_0}_\infty\|_\infty\leqslant e^{\lambda t_0}\|T^{-,c}_{nt_0+s_0}\va-\va^{s_0}_\infty\|_\infty.
				\]
				Thus, we get that
				\[
				+\infty=\lim_{n\to\infty}T^{-,c}_{t_0-s_0}\circ T^{-,c}_{nt_0+s_0}\va(x)=T^{-,c}_{t_0-s_0}\va^{s_0}_\infty(x),
				\]
				a contradiction. We have proved the assertion.

				Next, we prove for any $s\in[0,t_0]$,
				$\lim_{t\to+\infty}T^{-,c}_{nt_0+s}\varphi(x)=+\infty$ uniformly on $x\in M$.  Suppose not. Then there are $s'\in[0,s_0]$, $K_0>0$, $\{n_k\}\nearrow+\infty$ and $x_k\in M$, such that
				\[
				T^{-,c}_{n_kt_0+s'}\va(x_k)\leqslant K_0.
				\]
				Then for any $k\in \mathbb{N}$, any $x\in M$
				\[
				T^{-,c}_{(n_k+1)t_0+s'}\va(x)\leqslant h^c_{x_k,T^{-,c}_{n_kt_0+s'}\va(x_{n_k})}(x,t_0)\leqslant h^c_{x_k,K_0}(x,t_0)\leqslant \max_{y',y''\in M}h^c_{y',K_0}(y'',t_0)<+\infty,
				\]
				which contradicts $\lim_{n\to\infty}T^{-,c}_{nt_0+s'}\varphi(x)=+\infty$ for all $x\in M$.

				%
				\noindent {\it Case (ii) }:	
				If $\lim_{n\to+\infty}T^{-,c}_{nt_0}\va(x)=\va_\infty(x)$ for all $x\in M$, where $\va_\infty(x)$ is a Lipschitz continuous function on $M$, then for any $s\in\R$,
				\[
				\lim_{n\to+\infty}T^{-,c}_{nt_0+s}\va(x)=T^{-,c}_s\va_\infty(x)=:u(x,s).
				\]
				It is clear that $u(x,s+t_0)=u(x,s)$ for all $s\in\R$, and that $u(x,s)$ is a viscosity solution of (1.5a).
				
				The proof for the case  $\va\geqslant T^{-,c}_{t_0}\va$ for some $t_0>0$ is quite similar and thus we omit it.
				
				\medskip
				
				\noindent {\bf Step 2}:		
				if there is $t_0>0$ such that $\va<T^{-,c}_{t_0}\va$, then there is $t_1>0$ close enough to $t_0$ such that $t_1/t_0$ is an irrational number and $\va<T^{-,c}_{t_1}\va$. Since $t_1/t_0$ is an irrational number, then for any $s\in[0,t_0]$, any $\eps>0$ and any $N\in\mathbb{N}$, there are $m_0$, $m_1\in\mathbb{N}$ with $m_0$, $m_1>N$, such that
				\begin{align}\label{5-3}
					|m_1t_1-(m_0t_0+s)|<\eps.
				\end{align}
				By the result obtained in Proposition \ref{thD} (1), then $\lim_{n\to+\infty}T^{-,c}_{nt_0+s}\va(x)=+\infty$  for all $x\in M$ and all $s\in\R$, or $\lim_{n\to+\infty}T^{-,c}_{nt_0+s}\va(x)=u(x,s)$ for all $x\in M$ and all $s\in\R$. If $\lim_{n\to+\infty}T^{-,c}_{nt_0+s}\va(x)=u(x,s)$ for all $x\in M$ and all $s\in[0,t_0]$, then in view of \eqref{5-3}, we get that
				$\lim_{n\to\infty}T^{-,c}_{nt_1}\va(x)=:\va'_\infty(x)$ for some $\va'_\infty\in \mathrm{Lip}(M)$. By \eqref{5-3} again, $\va'_\infty(x)=u(x,s)$ for all $x\in M$ and $s\in[0,t_0]$. Thus, $\va'_\infty$ is a viscosity solution of \eqref{shjc}.
				
				The proof for the case  $\va\geqslant T^{-,c}_{t_0}\va$ for some $t_0>0$ is quite similar and thus we omit it.

				\medskip
				
				\noindent {\bf Step 3}:	
				if for any $t>0$, there are $x_1$, $x_2\in M$ such that $T^{-,c}_{t}\va(x_1)>\va(x_1)$ and $T^{-,c}_{t}\va(x_2)<\va(x_2)$, then for any $t>0$, there is $x_t\in M$ such that $T^{-,c}_{t}\va(x_t)=\va(x_t)$. Note that
				\[
				T^{-,c}_{t}\va(x)=T^{-,c}_{1}\circ T^{-,c}_{t-1}\va(x)\leqslant h_{x_{t-1},T^{-,c}_{t-1}\va(x_{t-1})}(x,1)= h_{x_{t-1},\va(x_{t-1})}(x,1).
				\]
				Thus, $T^{-,c}_{t}\va(x)$ is bounded from above. Note that for any $y\in M$,
				\[
				h^c_{y,T^{-,c}_t\va(y)}(x_{t+1},1)\geqslant T^{-,c}_{t+1}\va(x_{t+1})=\va(x_{t+1}),
				\]
				which implies that $T^{-,c}_t\va(y)\geqslant h_c^{x_{t+1},\va(x_{t+1})}(y,1)$ for all $y\in M$. Therefore, we get that for any $t>1$, any $y\in M$,
				\[
				\min_{(z,z')\in M\times M}h_c^{z,\va(z)}(z',1)	\leqslant h_c^{x_{t+1},\va(x_{t+1})}(y,1)\leqslant T^{-,c}_t\va(y)\leqslant\max_{(z,z')\in M\times M} h^c_{z,\va(z)}(z',1).
				\]
				Hence, $|T^{-,c}_t\va(y)|$ is bounded on $M\times[0,+\infty)$.
			\end{proof}

			\subsection{Necessary and sufficient conditions for the existence  II}
			We prove the following result in this part.
			\begin{prop}\label{thH}
				Let $c\in\R$. The following  statements are equivalent.
				\begin{itemize}
					\item [(1)] Equation \eqref{shjc} has viscosity solutions;
					\item [(2)] There exist $\varphi$, $\psi\in C(M,\R)$ and $t_1$, $t_2\in \R_+$ such that $T^{-,c}_{t_1}\varphi\geqslant \varphi$, $T^{-,c}_{t_2}\psi\leqslant \psi$;
					\item [(3)] There exist $\varphi$, $\psi\in C(M,\R)$ such that $T^{-,c}_{t}\varphi$ is bounded from below and $T^{-,c}_{t}\psi$ is bounded from above on $M\times[0,+\infty)$.
				\end{itemize}
			\end{prop}

			\begin{proof}
				Without any loss of generality, we assume that $c=0$. The strategy of our proof: we will prove the equivalence of items (2) and (3), and then the equivalence of items (3) and (1).

				\medskip
				
				\noindent {\bf Step 1}:	we first show (2)$\Leftrightarrow$(3).
				If condition (2) holds true, then in view of Proposition \ref{thD} (1), one can deduce that $T^{-}_{t}\varphi$ is bounded from below and $T^{-}_{t}\psi$ is bounded from above on $M\times[0,+\infty)$.
				
				In the rest of Step 1, we show that (3)$\Rightarrow$(2). Let $\varphi$, $\psi\in C(M,\R)$ be such that $T^{-}_{t}\varphi$ is bounded from below and $T^{-}_{t}\psi$ is bounded from above on $M\times[0,+\infty)$.
				
				If $T^{-}_{t}\varphi$ is also bounded from above, then by Proposition \ref{pr-v}, $\va_\infty(x)=\liminf_{t\to+\infty}\va(x)$ is a viscosity solution of \eqref{s0} and thus $\va_\infty=T^-_t\va_\infty$ for all $t\geqslant 0$. If $T^{-}_{t}\varphi$ is  unbounded from above, then from Proposition \ref{thD} (3) we can get that there is $t'>0$ such that
				\[
				\{x\in M: T^-_{t'}\va(x)=\va(x)\}=\emptyset.
				\]
				Thus, one can deduce that either $T^-_{t'}\va>\va$ or 	$T^-_{t'}\va<\va$. Since $T^{-}_{t}\varphi$ is bounded from below, by Proposition \ref{thD} (2) again we get that
				\[
				\lim_{t\to+\infty}T^-_t\va(x)=+\infty
				\]
				uniformly in $x\in M$. Hence, there is $t_1>0$ such that $T^-_{t_1}\va\geqslant\va$.

				If $T^{-}_{t}\psi$ is also bounded from below, then by Proposition \ref{pr-v}, $\psi_\infty(x)=\liminf_{t\to+\infty}\psi(x)$ is a viscosity solution of \eqref{s0} and thus $\psi_\infty=T^-_t\psi_\infty$ for all $t\geqslant 0$. If $T^{-}_{t}\psi$ is  unbounded from below, then from Proposition \ref{thD} (3) we can get that there is $t''>0$ such that
				\[
				\{x\in M: T^-_{t''}\psi(x)=\psi(x)\}=\emptyset.
				\]
				Thus, one can deduce that either $T^-_{t''}\psi>\psi$ or 	$T^-_{t''}\psi<\psi$. Since $T^{-}_{t}\psi$ is bounded from above, by Proposition \ref{thD} (2) again we get that
				\[
				\lim_{t\to+\infty}T^-_t\psi(x)=-\infty
				\]
				uniformly in $x\in M$. Hence, there is $t_2>0$ such that $T^-_{t_2}\psi\leqslant\psi$.

				\noindent {\bf Step 2}:	next  we show  (1)$\Leftrightarrow$(3).
				The fact that (1)$\Rightarrow$(3) is clear. It suffices to show that (3)$\Rightarrow$(1).
				
				From the proof of (3)$\Rightarrow$(2), we only need to discuss the case:
				\[
				\lim_{t\to+\infty}T^-_t\va(x)=+\infty,\quad \text{and}\quad  \lim_{t\to+\infty}T^-_t\psi(x)=-\infty.
				\]
				Let $u_\rho=\rho \va+(1-\rho)\psi$, $\rho\in[0,1]$ and let
				\[
				\rho_0=\inf\{\rho: \lim_{t\to+\infty}T^-_tu_\rho(x)=+\infty\}.
				\]
				Consider $T^-_tu_{\rho_0}$. If $T^-_tu_{\rho_0}$ is bounded on $M\times[0,+\infty)$, then by Proposition \ref{pr-v} again,  $\liminf_{t\to+\infty}T^-_tu_{\rho_0}$ is a viscosity solution of \eqref{s0}. If $T^-_tu_{\rho_0}$ is unbounded on $M\times[0,+\infty)$, then by similar arguments used in Step 1, one can deduce that either  $\lim_{t\to+\infty}T^-_tu_{\rho_0}=+\infty$ uniformly in $x\in M$ or $\lim_{t\to+\infty}T^-_tu_{\rho_0}=-\infty$ uniformly in $x\in M$.

				If $\lim_{t\to+\infty}T^-_tu_{\rho_0}=+\infty$ uniformly in $x\in M$, then there is $t_0>0$ such that $T^-_{t_0}u_{\rho_0}>u_{\rho_0}$. Thus, there is $\epsilon_0>0$ such that $T^-_{t_0}u_{\rho_0-\epsilon_0}>u_{\rho_0-\epsilon_0}$. Then by Proposition \ref{thD} (2) we get that either $\lim_{t\to+\infty}T^-_{t}u_{\rho_0-\epsilon_0}=u_\infty$ uniformly in $x\in M$, or $\lim_{t\to+\infty}T^-_{t}u_{\rho_0-\epsilon_0}=+\infty$ uniformly in $x\in M$. In view of the definition of $\rho_0$, we deduce that $\lim_{t\to+\infty}T^-_{t}u_{\rho_0-\epsilon_0}=u_\infty$ uniformly in $x\in M$ and $u_\infty$ is a viscosity solution of \eqref{s0}.

				If $\lim_{t\to+\infty}T^-_tu_{\rho_0}=-\infty$ uniformly in $x\in M$, then there is $t_0'>0$ such that $T^-_{t_0'}u_{\rho_0}<u_{\rho_0}$. Thus, there exists $\delta>0$ such that
				\[
				T^-_{t_0'}u_{\rho_0+\epsilon}<u_{\rho_0+\epsilon}
				\]
				for all $\epsilon\in(0,\delta)$. Then for any $\epsilon\in(0,\delta)$, either  $\lim_{t\to+\infty}T^-_tu_{\rho_0+\epsilon}=-\infty$ uniformly in $x\in M$ or $\lim_{t\to+\infty}T^-_tu_{\rho_0+\epsilon}=u'_\infty$ uniformly in $x\in M$, where $u'_\infty$ is a viscosity solution of \eqref{s0}. Recall that $\lim_{t\to+\infty}T^-_tu_{\rho_0}=-\infty$ uniformly in $x\in M$. Hence, by the definition of $\rho_0$ there must be $\epsilon_0\in(0,\delta)$ such that
				$\lim_{t\to+\infty}T^-_tu_{\rho_0+{\epsilon_0}}=u''_\infty$ uniformly in $x\in M$, where $u''_\infty$ is a viscosity solution of \eqref{s0}.
			\end{proof}

%


			\section{Analysis of the admissible set: min-max and max-min formulas }
			
			\subsection{The admissible set is an interval}
			Before proving Theorem \ref{thC}, we recall a approximation and regularity result of Lipschitz functions by Czarnecki and Rifford \cite{CR}.

			\begin{prop}\label{Ri}
				Let $f\in\mathrm{Lip}(M)$. Then there exists a sequence $\{f_{n}\}_{n \in \mathbb{N}}$ in $C^\infty(M,\R)$ such that
				$$
				\lim _{n \rightarrow+\infty}\left\|f_{n}-f\right\|_{\infty}=0
				$$
				and
				$$
				\lim _{n \rightarrow+\infty} d_{\mathrm{Haus }}\left(\operatorname{graph}\left(D f_{n}\right), \operatorname{graph}(\partial f)\right)=0.
				$$	
			\end{prop}	
			Here,  $\partial f(x)$ denotes Clarke's generalized gradient of $f$ at $x$:
			$$
			\partial f(x)=\operatorname{co}\left\{\zeta \mid \exists\left(x_{n}\right)_{n \in \mathbb{N}} \subset \operatorname{Dom}(D f), x_{n} \rightarrow x, Df(x_{n}) \rightarrow \zeta,\ n\to\infty\right\}
			$$
			which is non-empty by Rademacher's theorem. And
			$\operatorname{graph}(\partial f):=\left\{(x, p) \in T^*M \ :\  p \in \partial f(x)\right\}.
			$
			Let $S_1$ and $S_2$ be two non-empty closed subsets of $T^*M$,
			$$
			d_{\text {Haus }}\left(S_1, S_2\right):=\sup \left\{\sup _{(x,p) \in S_1} d_{S_2}(x,p), \sup _{(x,p) \in S_2} d_{S_1}(x,p)\right\}
			$$
			denotes the Hausdorff distance, where $d_{S}(x,p)=\inf _{(x',p') \in S}d((x,p),(x',p'))$.

			In view of Rademacher's theorem,  $M \backslash$ $\operatorname{Dom}(Du)$ is negligible. Since $\left\|D u(x)\right\|_{x}$ is bounded by the Lipschitz constant of $u$
			and $H$ is of class $C^3$, then the following equality is a direct consequence of Proposition \ref{Ri}. 
%
			
			\begin{align}\label{991}
				\inf_{u\in\mathrm{Lip}(M)}\sup_{x\in\mathrm{Dom}(Du)}H(x,u(x),Du(x))=\inf_{u\in\mathrm{SCL}^+(M)}\sup_{x\in\mathrm{Dom}(Du)}H(x,u(x),Du(x)).
			\end{align}
			We omit the proof.

			\medskip
			\begin{proof}[Proof of Theorem \ref{thC}]
				First we show that if $c<\x$, then equation \eqref{shjc} has no viscosity subsolutions. Assume by contradiction that there is a viscosity subsolution of equation \eqref{shjc}. By classical results on viscosity solutions, we have that $u$ is Lipschitz on $M$ and satisfies $H(x,u(x),Du(x))\leqslant c$ for a.e. $x\in M$. Thus, by \eqref{991}
				\[
				\x=\inf_{w\in\mathrm{Lip}(M)}\sup_{x\in\mathrm{Dom}(Dw)}H(x,w(x),Dw(x))\leqslant c,
				\]
				a contradiction.

				Next we show that if $c>\y$, then equation \eqref{shjc} has no viscosity solutions. If equation \eqref{shjc} admits a viscosity solution $u^*$, then by Proposition \ref{pr1} one can deduce that $v^*=\lim_{t\to+\infty}T^{+,c}_tu^*$ is a forward weak KAM solution of \eqref{shjc}. Thus, $-v^*$ is a viscosity solution of $H(x,-w(x),-Dw(x))=c$, which implies that $H(x,v^*(x),Dv^*(x))=c$ a.e. $x\in M$. Hence, in view of the semiconvexity of $v^*$ we get that
				\[
				\y= \sup_{u\in \mathrm{SCL^+}(M)}\inf_{x\in \mathrm{Dom}(Du)}H(x,u(x),Du(x))\geqslant c.
				\]

				Combining the above arguments and the non-emptiness of $\mathfrak{C}$, one have that
				
				\[
				\x\leqslant \y,\quad  \text{and}\quad \mathfrak{C}\subset [\x,\y].
				\]
				
				To finish the proof, it suffices to show that for any $c\in(\x,\y)$, equation \eqref{shjc} admits at least a viscosity solution. Since $c>\x$, then by definition there is $\va\in\mathrm{Lip}(M)$ such that $\va$ is a viscosity subsolution of \eqref{shjc}, and thus $\va\leqslant T^{-,c}_t\va$ for all $t\geqslant 0$.

				In view of Theorem \ref{thA} (4), we only need to prove that there is   $\psi\in C(M,\R)$ such that $T^{-,c}_t\psi$ is bounded from above. Since $c<\y$, then there are $\delta>0$ and $\psi\in\mathrm{SCL^+}(M)$ such that
				\[
				H(x,\psi(x),D\psi(x))-c\geqslant \delta,\quad \mathrm{a.e.}\ x\in M.
				\]
				If for any $t>0$, $\{x\in M: T^{-,c}_t\psi(x)=\psi(x)\}\neq \emptyset$, then $T^{-,c}_t\psi$ is bounded. If there is $t_0>0$ such that either $T^{-,c}_{t_0}\psi<\psi$, or $T^{-,c}_{t_0}\psi>\psi$, we only need to take care of the case $T^{-,c}_{t_0}\psi>\psi$. By Proposition \ref{thD} (2),
				either $\lim_{t\to+\infty}T^{-,c}_{t}\psi=\psi_\infty$, where $\psi_\infty$ is a viscosity solution of \eqref{shjc}, or $\lim_{t\to+\infty}T^{-,c}_{t}\psi=+\infty$. In the rest of the proof we show that the case $\lim_{t\to+\infty}T^{-,c}_{t}\psi=+\infty$ cannot happen.
				
				Assume by contradiction that $\lim_{t\to+\infty}T^{-,c}_{t}\psi=+\infty$. Thus, there are $t_1>0$ and $x_0\in M$ such that $T^{-,c}_{t_1}\psi\geqslant \psi$ for all $t\geqslant t_1$ and $T^{-,c}_{t_1}\psi(x_0)=\psi(x_0)$. Note that $\psi$ is a semiconvex function and $T^{-,c}_{t_1}\psi$ is a semiconcave function \cite[Theorem 3.2]{CS0}. Then by Lemma \ref{jc} in the Appendix, both $\psi':=T^{-,c}_{t_1}\psi$ and $\psi$ are differentiable at $x_0$, and $\psi(x_0)=\psi'(x_0)$, $D\psi(x_0)=D\psi'(x_0)$.
				Let $u_0=\psi(x_0)=\psi'(x_0)$, $p_0=D\psi(x_0)=D\psi'(x_0)$.
				Let $(x(t),u(t),p(t))$ be the solution of \eqref{c} with $(x(0),u(0),p(0))=(x_0,u_0,p_0)$ in a small neighbourhood of $0$. Hence, we have that
				\[
				u(t)=u_0+\big(\langle p_0,\dot{x}(0)\rangle_{x_0}-H(x_0,u_0,p_0)+c\big)t+\mathfrak{o}(t),
				\]
				and
				\[
				\psi(x(t))=u_0+\langle p_0,\dot{x}(0)\rangle_{x_0}t+\mathfrak{o}(t).
				\]
				Recall that
				\[
				H(x_0,u_0,p_0)-c\geqslant \delta,
				\]
				which implies that $\psi(x(t))>u(t)$ in a small enough neighbourhood of 0. Thus we have
				\[
				T^{-,c}_t\psi'(x(t))\leqslant h^c_{x_0,u_0}(x(t),t)\leqslant u(t)<\psi(x(t))
				\]
				for sufficiently small $t>0$, which contradicts that
				\[
				T^{-,c}_t\psi'\geqslant \psi,\quad \forall t\geqslant 0.
				\]

				The proof of Theorem \ref{thC} is complete.
			\end{proof}

			\subsection{Examples}
			Let us discuss several illustrative examples of contact Hamiltonians satisfying (H1)-(H3) and describe the corresponding $\x$, $\y$. We discuss genuine contact Hamiltonians in the first two examples, while a classical Hamiltonian is studied in the last example.
			The classical case can be regarded as the critical case.

			Let $h(x,p)$ denote a generic Tonelli Hamiltonian on $T^*M$ in the following.
			
		
		\begin{ex}[$\x=-\infty$, $\y=+\infty$]
			Let $H(x,u,p)=f(x)u+h(x,p)$ for all $(x,u,p)\in T^*M\times\R$, where $f$ is a smooth function on $M$.
			\begin{itemize}
				\item [(i)] If $f(x)>0$ for all $x\in M$, then for any $a<0$,
				\[
				\sup_{x\in M}(f(x)a+h(x,0))\leqslant \sup_{x\in M}(f(x)a)+\sup_{x\in M}h(x,0)= a\inf_{x\in M}f(x)+\sup_{x\in M}h(x,0).
				\]
				Letting $a\to-\infty$, we get that $\x=-\infty$.
				Similarly, for any $a>0$, we have that
				\[
				\inf_{x\in M}(f(x)a+h(x,0))\geqslant \inf_{x\in M}(f(x)a)+\inf_{x\in M}h(x,0)= a\inf_{x\in M}f(x)+\inf_{x\in M}h(x,0).
				\]
				Letting $a\to+\infty$, we get that $\y=+\infty$.
				\item [(ii)] For case $f(x)<0$ for all $x\in M$, one can get the same results in a similar manner.
			\end{itemize}
		\end{ex}
		
		\medskip
		
		\begin{ex}
			Let $H(x,u,p)=V(u)+h(x,p)$ for all $(x,u,p)\in T^*M\times\R$, where $V(u)$ is a smooth function on $\R$ and $\|V'\|_\infty\leqslant\lambda$.
			\begin{itemize}
				\item [(i)] ($\x\in\R$, $\y=+\infty$). Assume, in addition, $V$ is bounded from below and $\sup_{u\in\R}V(u)=+\infty$.
				Since $V(u)+h(x,p)$ is bounded from below, then
				\[
				\x=\inf_{u\in\mathrm{Lip}(M)}\sup_{x\in\mathrm{Dom}(Du)}V(u)+h(x,Du(x))>-\infty,
				\]
				and
				\[
				\x\leqslant V(a)+\sup_{x\in M}h(x,0),\quad \forall a\in\R.
				\]
				Thus, $\x\in\R$.  Note that,
				\[
				\sup_{a\in\R}V(a)+\inf_{x\in M}h(x,0)=+\infty.
				\]
				Thus, we get that
				\[
				\y=\sup_{u\in \mathrm{SCL^+}(M)}\inf_{x\in \mathrm{Dom}(Du)}\Big(V(u)+h(x,Du(x))\Big)=+\infty.
				\]

				\item [(ii)] ($\x=-\infty$, $\y\in\R$). Assume, in addition, $V$ is bounded from above and $\inf_{u\in\R}V(u)=-\infty$.
				\[
				\x\leqslant \inf_{a\in\R}(V(a)+\sup_{x\in M}h(x,0))= \inf_{a\in\R}V(a)+\sup_{x\in M}h(x,0)=-\infty.
				\]
				Notice that for any $u\in \mathrm{SCL^+}(M)$,
				\[
				\inf_{x\in M}(V(u(x))+h(x,Du(x)))\leqslant \sup_{x\in M}V(u(x))+\inf_{x\in M}h(x,Du(x))\leqslant\sup_{x\in M}V(u(x))+h(y,0),
				\]
				where $y$ is an arbitrary point in $M$ with $Du(y)=0$.
				Hence, we deduce that $\y\in\R$.
			\end{itemize}
		\end{ex}

		\medskip
		
		\begin{ex}[$\x=\y=0$]
			Let $H(x,u,p)=\|p\|_x^2$ for all $(x,u,p)\in T^*M\times\R$.
			For each $u\in \mathrm{SCL^+}(M)$,
			\[
			\inf_{x\in M}\|Du(x)\|_x^2=0,
			\]
			which implies that
			\[
			\y=\sup_{u\in \mathrm{SCL^+}(M)}\inf_{x\in M}\|Du(x)\|_x^2=0.
			\]
			By definition, it is direct to see that
			\[
			\x=\inf_{u\in \mathrm{SCL^+}(M)}\sup_{x\in M}\|Du(x)\|_x^2=0.
			\]
			In this example, $\mathfrak{C}=\{0\}$.
		\end{ex}

\section{Appendix}

\subsection{Proof of Proposition \ref{pr2.3}}
In order to prove the proposition \ref{pr2.3}, we provide a preliminary lemma.
\begin{lem}\label{lem2.1}
	Given any $x_0$, $x\in M$, $u_0\in\mathbb{R}$ and $t>0$, let $\gamma:[0,t]\to M$ be a minimizer of $h_{x_0,u_0}(x,t)$. Then for each $t_0\in(0,t)$, there is a unique minimizer of $h_{x_0,u_0}(\gamma(t_0),t_0)$.
\end{lem}

\begin{proof}
	Since $\gamma$ is a minimizer of $h_{x_0,u_0}(x,t)$, then $\gamma\big|_{[0,t_0]}$ is a minimizer $h_{x_0,u_0}(\gamma(t_0),t_0)$. If there is another minimizer $h_{x_0,u_0}(\gamma(t_0),t_0)$, denoted by $\alpha$, then we will show that $\alpha=\gamma\big|_{[0,t_0]}$. Let
	\[
	\beta(s):=\left\{\begin{array}{ll}
		\alpha(s), \quad s\in[0,t_0],\\
		\gamma(s),\quad s\in[t_0,t].
	\end{array}\right.
	\]
	Then we get
	\begin{align*}
		h_{x_0,u_0}(x,t)&=h_{x_0,u_0}(\gamma(t_0),t_0)+\int_{t_0}^tL(\gamma(s),h_{x_0,u_0}(\gamma(s),s),\dot{\gamma}(s))ds\\
		&=h_{x_0,u_0}(\alpha(t_0),t_0)+\int_{t_0}^tL(\gamma(s),h_{x_0,u_0}(\gamma(s),s),\dot{\gamma}(s))ds\\
		&=u_0+\int_0^{t_0}L(\alpha(s),h_{x_0,u_0}(\alpha(s),s),\dot{\alpha}(s))ds\\
		&\ \ \ \ \ \ \ \ \  +\int_{t_0}^tL(\gamma(s),h_{x_0,u_0}(\gamma(s),s),\dot{\gamma}(s))ds\\
		&=u_0+\int_0^tL(\beta(s),h_{x_0,u_0}(\beta(s),s),\dot{\beta}(s))ds,
	\end{align*}
	which implies that $\beta$ is a minimizer of $h_{x_0,u_0}(x,t)$. From
	Proposition  \ref{IVP}, $\gamma$ and $\beta$ are both of class $C^1$. Therefore, we have $\dot{\gamma}(t_0)=\dot{\beta}(t_0)$. By Proposition  \ref{IVP} and the uniqueness of solutions of initial value problem of ordinary differential equations, we have $\alpha(s)=\gamma(s)$ for all $s\in[0,t_0]$, which completes the proof.
\end{proof}

\medskip

\begin{proof}[Proof of Proposition \ref{pr2.3}]
	We divide the proof in two steps.
	
	\medskip
	
	\noindent {\bf Step 1}: Given any $t_1$, $t_2\in\mathbb{R}$ with $t_1< t_2$ and $t_0\in(t_1,t_2)$, since  $(x(t),u(t))$ is  globally minimizing, then we have
	\begin{align*}
		u(t_2)&=h_{x(t_1),u(t_1)}(x(t_2),t_2-t_1),\\
		u(t_2)&=h_{x(t_0),u(t_0)}(x(t_2),t_2-t_0),\\
		u(t_0)&=h_{x(t_1),u(t_1)}(x(t_0),t_0-t_1).
	\end{align*}
	It follows that
	\begin{align*}
		h_{x(t_1),u(t_1)}(x(t_2),t_2-t_1)&=h_{x(t_0),u(t_0)}(x(t_2),t_2-t_0)
		=h_{x(t_0),h_{x(t_1),u(t_1)}(x(t_0),t_0-t_1)}(x(t_2),t_2-t_0).
	\end{align*}
	In view of Proposition \ref{pr-af}(4), there is a minimizer of $h_{x(t_1),u(t_1)}(x(t_2),t_2-t_1)$, denoted by $\gamma$, such that $\gamma(t_0)=x(t_0)$.
	
	\medskip
	
	\noindent{\bf Step 2}: From the above arguments, there exists a minimizer $\alpha$ of $h_{x(t_1),u(t_1)}(x(t_2+1),t_2-t_1+1)$ such that $x(t_2)=\alpha(t_2)$. By Lemma \ref{lem2.1}, $\alpha\big|_{[t_1,t_2]}$ is the unique minimizer of $h_{x(t_1),u(t_1)}(x(t_2),t_2-t_1)$. By the arguments used in Step 1 again, $x(s)=\alpha(s)$ for all $s\in[t_1,t_2]$. Thus, by Proposition  \ref{IVP} and the arbitrariness of $t_1$ and $t_2$ with $t_1<t_2$, $x(t)$ is of class $C^1$ for $t\in\mathbb{R}$, and $(x(t),u(t),p(t))$ is a solution of (\ref{c}), where $p(t):=\frac{\partial L}{\partial \dot{x}}(x(t),u(t),\dot{x}(t))$.
	Since
	\[
	\dot{u}(t)=L(x(t),u(t),\dot{x}(t)),
	\]
	it is easy to see that $x(t)|_{[t_1,t_2]}$ is a minimizer of $h_{x(t_1),u(t_1)}(x(t_2),t_2-t_1)$.
	
\end{proof}

\subsection{Proof of Proposition \ref{iv}}

\begin{lem}\label{ulipp}
	If $\varphi\prec L$, then $\varphi$ is Lipschitz continuous on $M$.
\end{lem}
\begin{proof}
	For each $x$, $y\in M$, let $\gamma:[0,d(x,y)]\to M$ be a geodesic of length $d(x,y)$, parameterized by arclength and connecting $x$ to $y$. Since $M$ is compact and $\varphi$ is continuous, then
	\[
	A_1:=\max_{x\in M}|\varphi(x)|\quad A_2:=\sup\{L(x,u,\dot{x})\ |\ x\in M,\ |u|\leqslant A_1,\ \|\dot{x}\|_x=1\}
	\]
	are well-defined.
	Since $\|\dot{\gamma}(s)\|_{\gamma(s)}=1$ for each $s\in[0,d(x,y)]$, we have $L(\gamma(s),\varphi(\gamma(s)),\dot{\gamma}(s))\leqslant A_2$. Then by $\varphi\prec L$,
	\begin{align*}
		\varphi(\gamma(d(x,y)))-\varphi(\gamma(0))\leqslant \int_0^{d(x,y)}L(\gamma(s),\varphi(\gamma(s)),\dot{\gamma}(s))ds
		\leqslant \int_0^{d(x,y)}A_2ds=A_2d(x,y).
	\end{align*}
	We finish the proof by exchanging the roles of $x$ and $y$.
\end{proof}

\medskip

\begin{lem}\label{labell}
	Let $\varphi\prec L$ and let $\gamma:[a,b]\rightarrow M$ be a $(\varphi,L,0)$-calibrated curve. If $\varphi$ is differentiable at $\gamma(t)$ for some $t\in (a,b)$, then we have
	\[
	H\left(\gamma(t),\varphi(\gamma(t)),D\varphi(\gamma(t))\right)=0,\qquad D\varphi(\gamma(t))=\frac{\partial L}{\partial \dot{x}}(\gamma(t),\varphi(\gamma(t)),\dot{\gamma}(t)).
	\]
\end{lem}

\begin{proof}
	By Lemma \ref{ulipp}, $\varphi$ is Lipschitz continuous on $M$. We first show that at each point $x\in M$ where $D\varphi(x)$ exists, we have
	\begin{align}\label{4-100}
		H(x,\varphi(x),D\varphi(x))\leqslant 0.
	\end{align}
	For any given $\varphi\in T_xM$, let $\alpha: [0,1]\rightarrow M$ be a $C^1$ curve such that $\alpha(0)=x$, $\dot{\alpha}(0)=v$. By $\varphi\prec L$, for each $t\in [0,1]$, we have
	\[
	\varphi(\alpha(t))-\varphi(\alpha(0))\leqslant \int_0^tL(\alpha(s),\varphi(\alpha(s)),\dot{\alpha}(s)))ds.
	\]
	Dividing by $t>0$ and let $t\rightarrow 0^+$, we have
	$\langle D\varphi(x),v\rangle\leqslant L(x,\varphi(x),v)$,
	which implies $	H(x,\varphi(x),D\varphi(x))=\sup_{v\in T_xM}(\langle D\varphi(x),v\rangle_x-L(x,\varphi(x),v))\leqslant 0$. Thus, (\ref{4-100}) holds.

	If $\varphi$ is differentiable at $\gamma(t)$ for some $t\in (a,b)$, then for each $t'\in [a,b]$ with $t\leqslant t'$, we have
	$\varphi(\gamma(t'))-\varphi(\gamma(t))= \int_t^{t'}L(\gamma(s),\varphi(\gamma(s)),\dot{\gamma}(s))ds$,
	since $\gamma:[a,b]\rightarrow M$ is a $(\varphi,L,0)$-calibrated curve.
	Dividing by $t'-t$ and let $t'\rightarrow t^+$, we have
	$\langle D\varphi(\gamma(t)),\dot{\gamma}(t)\rangle_{\gamma(t)}= L(\gamma(t),\varphi(\gamma(t)),\dot{\gamma}(t))$.
	Thus, we have
	\[
	H(\gamma(t),\varphi(\gamma(t)),D\varphi(\gamma(t)))\geqslant \langle D\varphi(\gamma(t)),\dot{\gamma}(t)\rangle_{\gamma(t)}-L(\gamma(t),\varphi(\gamma(t)),\dot{\gamma}(t))=0,
	\]
	which together with (\ref{4-100}) implies  $H(\gamma(t),\varphi(\gamma(t)),D\varphi(\gamma(t)))=0$ and
	\[
	\langle D\varphi(\gamma(t)),\dot{\gamma}(t)\rangle_{\gamma(t)}=H(\gamma(t),\varphi(\gamma(t)),D\varphi(\gamma(t)))+L(\gamma(t),\varphi(\gamma(t)),\dot{\gamma}(t)).
	\]
	In view of Legendre transform, we get
	\[
	D\varphi(\gamma(t))=\frac{\partial L}{\partial \dot{x}}(\gamma(t),\varphi(\gamma(t)),\dot{\gamma}(t)).
	\]
	This completes the proof.
\end{proof}

\medskip

\begin{lem}\label{diffe}
	Given any $a>0$, let $\varphi\prec L$ and let $\gamma:[-a,a]\rightarrow M$ be a $(\varphi,L,0)$-calibrated curve. Then $\varphi$ is differentiable at $\gamma(0)$.
\end{lem}

\begin{proof}
	It suffices to prove the lemma for the case when $M=U$ is an open subset of $\mathbb{R}^n$. Set $x=\gamma(0)$. In order to prove the differentiability of $u$ at $x$, we only need to show for each $y\in U$, there holds
	\begin{equation}\label{supinff}
		\limsup_{\lambda\rightarrow 0^+}\frac{\varphi(x+\lambda y)-\varphi(x)}{\lambda}\leqslant \frac{\partial L}{\partial \dot{x}}(x,\varphi(x),\dot{\gamma}(0))\cdot y
		\leqslant \liminf_{\lambda\rightarrow 0^+}\frac{\varphi(x+\lambda y)-\varphi(x)}{\lambda}.
	\end{equation}
	For $\lambda>0$ and $0<\varepsilon\leqslant a$, define $\gamma_\lambda:[-\varepsilon,0]\rightarrow U$ by $\gamma_\lambda(s)=\gamma(s)+\frac{s+\varepsilon}{\varepsilon}\lambda y$. Then $\gamma_\lambda(0)=x+\lambda y$ and $\gamma_\lambda(-\varepsilon)=\gamma(-\varepsilon)$. Since $u\prec L$ and $\gamma:[-a,a]\rightarrow M$ is a $(\varphi,L,0)$-calibrated curve,  we have
	\[
	\varphi(x+\lambda y)-\varphi(\gamma(-\varepsilon))\leqslant \int_{-\varepsilon}^{0}L(\gamma_\lambda(s),\varphi(\gamma_\lambda(s)),\dot{\gamma}_\lambda(s))ds,
	\]
	and
	\[
	\varphi(x)-\varphi(\gamma(-\varepsilon))= \int_{-\varepsilon}^{0}L(\gamma(s),\varphi(\gamma(s)),\dot{\gamma}(s))ds.
	\]
	It follows that
	\[
	\frac{\varphi(x+\lambda y)-\varphi(x)}{\lambda}\leqslant\frac{1}{\lambda}\int_{-\varepsilon}^{0}\Big(L(\gamma_\lambda(s),\varphi(\gamma_\lambda(s)),\dot{\gamma}_\lambda(s))-L(\gamma(s),\varphi(\gamma(s)),\dot{\gamma}(s))\Big)ds.
	\]
	By Lemma \ref{ulipp}, there exists $K>0$ such that

	\[
	|\varphi(\gamma_\lambda(s))-\varphi(\gamma(s))|\leqslant  K\|\gamma_\lambda(s)-\gamma(s)\|=K\cdot\frac{s+\varepsilon}{\varepsilon}\cdot\lambda \|y\|,
	\]
	which implies
	\begin{align*}
		\limsup_{\lambda\rightarrow 0^+}\frac{\varphi(x+\lambda y)-\varphi(x)}{\lambda}&\leqslant\int_{-\varepsilon}^{0}\Big(\frac{s+\varepsilon}{\varepsilon}\cdot\frac{\partial L}{\partial x}(\gamma(s),\varphi(\gamma(s)),\dot{\gamma}(s))\cdot y\\
		&+K\frac{s+\varepsilon}{\varepsilon}|\frac{\partial L}{\partial \varphi}(\gamma(s),\varphi(\gamma(s)),\dot{\gamma}(s))|\|y\|\\
		&+\frac{1}{\varepsilon}\frac{\partial L}{\partial \dot{x}}(\gamma(s),\varphi(\gamma(s)),\dot{\gamma}(s))\cdot y \Big)ds.
	\end{align*}
	If we let $\varepsilon\rightarrow 0^+$, we get the first inequality in (\ref{supinff}).

	Define $\gamma_\lambda:[0,\varepsilon]\rightarrow M$ by $\gamma_\lambda(s)=\gamma(s)+\frac{\varepsilon-s}{\varepsilon}\lambda y$. We have
	\begin{align*}
		\varphi(\gamma(\varepsilon))-\varphi(x+\lambda y)&\leqslant \int^{\varepsilon}_{0}L(\gamma_\lambda(s),\varphi(\gamma_\lambda(s)),\dot{\gamma}_\lambda(s))ds,\\
		\varphi(\gamma(\varepsilon))-\varphi(x)&= \int^{\varepsilon}_{0}L(\gamma(s),\varphi(\gamma(s)),\dot{\gamma}(s))ds.
	\end{align*}
	It follows that
	\[
	\frac{\varphi(x+\lambda y)-\varphi(x)}{\lambda}\geqslant\frac{1}{\lambda}\int^{\varepsilon}_{0}\Big(L(\gamma(s),\varphi(\gamma(s)),\dot{\gamma}(s))-L(\gamma_\lambda(s),\varphi(\gamma_\lambda(s)),\dot{\gamma}_\lambda(s))\Big)ds,
	\]
	which implies
	\begin{align*}
		\liminf_{\lambda\rightarrow 0^+}\frac{\varphi(x+\lambda y)-\varphi(x)}{\lambda}&\geqslant\int^{\varepsilon}_{0}\Big(\frac{s-\varepsilon}{\varepsilon}\frac{\partial L}{\partial x}(\gamma(s),\varphi(\gamma(s)),\dot{\gamma}(s))\cdot y\\
		&+K\frac{s-\varepsilon}{\varepsilon}|\frac{\partial L}{\partial \varphi}(\gamma(s),\varphi(\gamma(s)),\dot{\gamma}(s))|\|y\|\\
		&+\frac{1}{\varepsilon}\frac{\partial L}{\partial \dot{x}}(\gamma(s),\varphi(\gamma(s)),\dot{\gamma}(s))\cdot y\Big)ds.
	\end{align*}
	Letting $\varepsilon\rightarrow 0^+$, we obtain the second inequality in (\ref{supinff}).
	This completes the proof.
\end{proof}

\medskip

\begin{proof}[ Proof of Proposition \ref{iv}]
	Let ${\bf{u}}(t):=u(\gamma(t))$ for $t\leqslant 0$. We assert that
	for each $s$, $t<0$ with $s<t$, there holds
	\begin{align}\label{neggmin}
		{\bf{u}}(t)=h_{\gamma(s),{\bf{u}}(s)}(\gamma(t),t-s).
	\end{align}

	If the assertion is true, then by Proposition \ref{pr2.3},  $\big(\gamma(t),{\bf{u}}(t),p(t)\big)$ satisfies equations (\ref{c}) on $(-\infty,0)$, where $p(t)=\frac{\partial L}{\partial \dot{x}}(\gamma(t),{\bf{u}}(t),\dot{\gamma}(t))$. Now we prove the assertion. Since $u$ is a backward weak KAM solution, then we have $
	T^-_{\sigma}u(x)=u(x), \forall x\in M, \forall \sigma\geq 0.$
	Recall  that $T^-_{\sigma}u(x)=\inf_{y\in M}h_{y,u(y)}(x,{\sigma})$ for all $\sigma>0$. Given any $s<t\leqslant 0$, we get
	\begin{align}\label{4-101}
		{\bf{u}}(\tau)\leqslant h_{\gamma(s),{\bf{u}}(s)}(\gamma(\tau),\tau-s),\quad \forall \tau\in (s,t].
	\end{align}
	Since 	$\gamma:(-\infty,0]\rightarrow M$ is a  $(u,L,0)$-calibrated curve, then we have
	\[
	{\bf{u}}(t)-{\bf{u}}(s)=\int_s^tL(\gamma(\tau),{\bf{u}}(\tau),\dot{\gamma}(\tau))d\tau,
	\]
	which together with (\ref{4-101}) implies
	\[
	{\bf{u}}(t)\geqslant {\bf{u}}(s)+\int_s^tL(\gamma(\tau),h_{\gamma(s),{\bf{u}}(s)}(\gamma(\tau),\tau-s),\dot{\gamma}(\tau))d\tau\geqslant h_{\gamma(s),{\bf{u}}(s)}(\gamma(t),t-s).
	\]
	By (\ref{4-101}) again, we have
	${\bf{u}}(t)=h_{\gamma(s),{\bf{u}}(s)}(\gamma(t),t-s).$
	Hence, (\ref{neggmin}) holds.

	By Lemma \ref{labell} and Lemma \ref{diffe}, $u$ is differentiable at $\gamma(t)$ for any $t<0$ and
	\[
	Du(\gamma(t))=\frac{\partial L}{\partial \dot{x}}(\gamma(t),u(\gamma(t)),\dot{\gamma}(t)).
	\]
	Hence,
	$
	(\gamma(t+s),u(\gamma(t+s)),Du(\gamma(t+s)))=\Phi_{s}(\gamma(t),u(\gamma(t)),Du(\gamma(t)))$, $\forall t,\ s<0.$
	In view of Lemma \ref{labell}, we have
	\[
	H\big(\gamma(t),u(\gamma(t)),\frac{\partial L}{\partial \dot{x}}(\gamma(t),u(\gamma(t)),\dot{\gamma}(t))\big)=0,\quad \forall t<0,
	\]
	which completes the proof.
\end{proof}

\subsection{Proof of Lemma \ref{jc}}

\begin{lem}\label{jc}
	Let $\va\in\mathrm{SCL^-}(M)$ and $\psi\in\mathrm{SCL^+}(M)$. Let $x_0$ be a local minimum point of $\va-\psi$. Then both $\va$ and $\psi$ are differentiable at $x_0$ with $D\va(x_0)=D\psi(x_0)$.
\end{lem}

\begin{proof}
	Since $\va$, $-\psi$, $\va-\psi\in \mathrm{SCL^-}(M)$, then $D^+\va(x_0)$, $D^+(-\psi)(x_0)$, $D^+(\va-\psi)(x_0)$ are non-empty. Since $x_0$ is a local minimum point of $\va-\psi$, then
	\[
	D^+\va(x_0)+D^+(-\psi)(x_0)\subset D^+(\va-\psi)(x_0)=D(\va-\psi)(x_0)=\{0\},
	\]
	which implies that both $D^+\va(x_0)$ and $D^+(-\psi)(x_0)$ are singletons.
\end{proof}

\bigskip

\bigskip

\textbf{This paper has no associated data.}

\bigskip

\bigskip
\noindent {\bf Acknowledgements:}
Kaizhi Wang is supported by NSFC Grant No. 12171315, 11931016 and by Natural Science Foundation of Shanghai No. 22ZR1433100. Jun Yan is supported by NSFC Grant No. 12171096, 11790272. The authors thank Liang Jin for providing the proof of Lemma \ref{jc}.

\medskip


\begin{thebibliography}{99}\small
	
	
	\bibitem{Ar}
	M.-C. Arnaud, Pseudographs and the Lax-Oleinik semi-group: a geometric and dynamical interpretation, Nonlinearity \textbf{24} (2011), 71--78.
	
	
	\bibitem{Bar} M. Bardi and I. Capuzzo-Dolcetta, Optimal control and viscosity solutions of Hamilton-Jacobi-Bellman equations. Systems \& Control: Foundations \& Applications. Birkh\"auser, 1997.
	
	
	\bibitem{Bs} G. Barles and P. E. Souganidis, On the large time behavior of solutions of Hamilton-Jacobi equations, SIAM J. Math. Anal. \textbf{31} (2000), 925--939.
	
	
	\bibitem{Be} P. Bernard, Existence of $C^{1,1}$ critical sub-solutions of the Hamilton-Jacobi equation on compact manifolds, Ann. Sci. \'Ecole Norm. Sup. \textbf{40} (2007), 445--452.
	
	
	
	\bibitem{Be2}  P. Bernard, The dynamics of pseudographs in convex Hamiltonian systems, J. Amer. Math. Soc.  \textbf{21} (2008), 615--669.
	
	
	\bibitem{CS0}P. Cannarsa and H. Soner, Generalized one-sided estimates for solutions of 
Hamilton-Jacobi equations and applications, Nonlinear Anal. \textbf{13} (1989), 305--323.
	
	
	
	\bibitem{CS} P. Cannarsa and C. Sinestrari, Semiconcave functions, Hamilton-Jacobi equations, and optimal control. Progress in Nonlinear Differential Equations and their Applications, \textbf{58}, Birkh$\rm{\ddot{a}}$user Boston, Inc., Boston, MA, 2004.	
	
	
	\bibitem{CCJWY}
	P. Cannarsa, W. Cheng, L. Jin, K. Wang and J. Yan, Herglotz' variational principle and Lax-Oleinik evolution,
	J. Math. Pures Appl. \textbf{141} (2020), 99--136.
	
	
	
	\bibitem{Ci}
	G. Contreras, R. Iturriaga, G. Paternain and M. Paternain, Lagrangian graphs, minimizing measures and Ma\~n\'e's critical values, Geom. Funct. Anal. \textbf{8} (1998),  788--809.
	
	
	
	\bibitem{Con}
	G. Contreras, Action potential and weak KAM solutions, Calc. Var. Partial Differential Equations \textbf{13} (2001), 427--458.
	
	
	
	\bibitem{CL1} M. Crandall and P.-L. Lions, Viscosity solutions of Hamilton-Jacobi equations, Trans. Amer. Math. Soc. \textbf{277} (1983), 1--42.
	
	
	
	\bibitem{CL2}
	M. Crandall, L. Evans and  P.-L. Lions,  Some properties of viscosity solutions of Hamilton-Jacobi equations, Trans. Amer. Math. Soc. \textbf{282} (1984),  487--502.
	
	
	
	
	
	
	
	
	\bibitem{CR}
	M.-O. Czarnecki and L. Rifford, Approximation and regularization of Lipschitz functions: convergence of the gradients, Trans. Amer. Math. Soc. \textbf{358} (2006), 4467--4520.
	
	
	\bibitem{DS}
	A. Davini and A. Siconolfi, A generalized dynamical approach to the large time behavior of solutions of Hamilton-Jacobi equations, SIAM J. Math. Anal. \textbf{38} (2006),  478--502.
	
	
	\bibitem{DFIZ}
	A. Davini, A. Fathi, R. Iturriaga and M. Zavidovique, Convergence of the solutions of the discounted Hamilton-Jacobi equation: convergence of the discounted solutions, Invent. Math. \textbf{206} (2016), 29--55.
	
	
	
	\bibitem{EP}
	M. Entov, L. Polterovich, Contact topology and non-equilibrium
	thermodynamics, arXiv:2101.03770v2.
	
	
	\bibitem{E} L. Evans, Periodic homogenisation of certain fully nonlinear partial differential equations, Proc. Roy. Soc. Edinburgh Sect. A \textbf{206} (1992), 245--265.
	
	
	
	
	
	
	
	
	
	
	
	
	
	\bibitem{Fat97a}
	A. Fathi, Th\'eor\`eme KAM faible et th\'eorie de Mather sur les syst\`emes lagrangiens,  C. R.
	Acad. Sci. Paris S\'er. I Math. \textbf{324} (1997), 1043--1046.
	
	
	
	\bibitem{Fat97b}
	A. Fathi, Solutions KAM faibles conjugu\'ees et
	barri\`eres de Peierls, C. R. Acad. Sci. Paris
	S\'er. I Math. \textbf{325} (1997), 649--652.
	
	
	
	
	\bibitem{Fat98}
	A. Fathi, Sur la convergence du semi-groupe de Lax-Oleinik, C. R. Acad. Sci. Paris S\'er. I Math. \textbf{327} (1998), 267--270.
	
	
	
	\bibitem{Fat-b}
	A. Fathi, Weak KAM Theorem in Lagrangian Dynamics,
	\url{ www.math.u-bordeaux.fr/~pthieull/Recherche/KamFaible/Publications/Fathi2008_01.pdf}
	
	
	
	\bibitem{FS} A. Fathi and A. Siconolfi,  Existence of $C^1$ critical subsolutions of the Hamilton-Jacobi equation, Invent. math. \textbf{155} (2004), 363--388.
	
	
	
	
	\bibitem{FR} A. Figalli and  L. Rifford. {\it Aubry sets, Hamilton-Jacobi equations, and the Ma\~n\'e conjecture}. Geometric analysis, mathematical relativity, and nonlinear partial differential equations 599 (2012): 83--104.
	
	
	
	\bibitem{FG}
	A. Figalli, D. Gomes and D. Marcon, Weak KAM theory for a weakly coupled system of Hamilton-Jacobi equations, Calc. Var. Partial Differential Equations \textbf{55} (2016), no. 4, Art. 79, 32 pp.
	
	
	
	
	
	\bibitem{Gom}
	D. Gomes, Generalized Mather problem and selection principles for viscosity solutions and Mather measures, Adv. Calc. Var. \textbf{1} (2008), 291--307. 	
	
	
	
	
	
	
	\bibitem{H} G. Herglotz, Ber\"{u}hrungstransformationen, Lectures at the University of G\"{o}ttingen, G\"{o}ttingen, 1930.
	
	
	
	\bibitem{ishii}
	H. Ishii,  Asymptotic solutions for large time of Hamilton-Jacobi equations, International Congress of Mathematicians, Vol. III, 213--227, Eur. Math. Soc., Z\" urich, 2006.
	
	
	\bibitem{j}
	W. Jing, H. Mitake and H. Tran, Generalized ergodic problems: existence and uniqueness structures of solutions, J. Differential Equations \textbf{268} (2020), 2886--2909.
	
	
	
	
	\bibitem{Ka}
	V. Kaloshin, Mather theory, weak KAM theory, and viscosity solutions of Hamilton-Jacobi PDE's. (English summary) EQUADIFF 2003, 39-48, World Sci. Publ., Hackensack, NJ, 2005.
	
	
	
	
	
	\bibitem{Lb}
	P.-L. Lions,  Generalized solutions of Hamilton-Jacobi equations. Pitman, Boston, 1982.
	
	
	
	\bibitem{LPV}
	P.-L. Lions, G. Papanicolaou and S. R. S. Varadhan, Homogenization of Hamilton-Jacobi Equations, unpublished.
	
	
	
	\bibitem{MV}
	E. Maderna and A. Venturelli, Viscosity solutions and hyperbolic motions: a new PDE method for the N-body problem, Ann. of Math. (2) \textbf{192} (2020), 499--550.
	
	
	\bibitem{MS} S. Mar\`o and A. Sorrentino, Aubry-Mather theory for conformally symplectic systems, Commun. Math. Phys.  \textbf{354} (2017), 775--808.
	
	
	\bibitem{M} J. Mather, Action minimizing invariant measures for positive definite Lagrangian systems, Math. Z. \textbf{207} (1991), 169--207.
	
	
	
	\bibitem{Ms}
	H. Mitake and K. Soga, Weak KAM theory for discounted Hamilton-Jacobi equations and its application, Calc. Var. Partial Differential Equations \textbf{57} (2018), Paper No. 78, 32 pp.
	
	
	
	\bibitem{nr}
	G. Namah and J.-M. Roquejoffre, Remarks on the long time behaviour of the solutions of Hamilton-Jacobi equations,  Comm. Partial Differential Equations \textbf{24} (1999), 883--893.
	
	
	
	\bibitem{r}
	J.-M. Roquejoffre, Convergence to steady states or periodic solutions in a class of Hamilton-Jacobi equations, J. Math. Pures Appl. \textbf{80} (2001), 85--104.
	
	
	\bibitem{S}
	A. Sorrentino, Action-minimizing methods in Hamiltonian dynamics. An introduction to Aubry-Mather theory. Mathematical Notes, 50. Princeton University Press, Princeton, NJ, 2015. xii+115 pp.
	
	
	
	\bibitem{S1}	A. Sorrentino, On John Mather's seminal contributions in Hamiltonian dynamics, Methods Appl. Anal.  \textbf{26} (2019), 37--63. 	
	
	
	
	\bibitem{Tr}
	H. Tran, Hamilton-Jacobi equations: viscosity and applications,
	\url{https://people.math.wisc.edu/~hung/HJ%20equations-viscosity%20solutions%20and%20applications-v2.pdf}
	
	
	
	
	\bibitem{WY}
	K. Wang and J. Yan, A new kind of Lax-Oleinik type operator with parameters for time-periodic positive definite Lagrangian systems, Comm. Math. Phys. \textbf{309} (2012),  663--691.	
	
	
	\bibitem{WWY1} K. Wang, L. Wang and J. Yan, Implicit variational principle for contact Hamiltonian systems, Nonlinearity \textbf{30} (2017), 492--515.
	
	
	
	\bibitem{WWY2} K. Wang, L. Wang and J. Yan,  Variational principle for contact Hamiltonian systems and its applications, J. Math. Pures Appl. \textbf{123} (2019), 167--200.
	
	
	\bibitem{WWY3} K. Wang, L. Wang and J. Yan, Aubry-Mather theory for contact Hamiltonian systems, Commun. Math. Phys. \textbf{366} (2019), 981--1023.	
	
	
	
	\bibitem{WWY4}
	K. Wang, L. Wang and J. Yan, Weak KAM solutions of Hamilton-Jacobi equations with decreasing dependence on unknown functions, J. Differential Equations \textbf{286} (2021), 411--432.
	
	
	
	
	
	
	
	
	
	
	
	
	
	
	
	
	
	
	
	
	
\end{thebibliography}
\end{document}